\documentclass[a4paper]{article}

\usepackage{amsmath, amsthm, amsfonts, amssymb, accents}
\usepackage{graphicx,color}
\usepackage{srcltx}
\usepackage{dsfont}

\theoremstyle{plain}
\newtheorem{theorem}{Theorem}[section]

\newtheorem{lemma}[theorem]{Lemma}

\theoremstyle{remark}

\numberwithin{equation}{section}

\def\tht{\theta}
\def\Om{\Omega}

\def\e{\varepsilon}
\def\g{\gamma}
\def\G{\Gamma}

\def\p{\partial}
\def\D{\Delta}

\def\E{\mbox{\rm e}}
\def\a{\alpha}

\def\d{\delta}

\def\vp{\varphi}

\def\Odr{\mathcal{O}}
\def\H{W_2}

\def\Hinf{W_\infty}

\def\Ho{W_{2,0}}

\def\di{\,d}

\def\iu{\mathrm{i}}

\def\Hd{\mathcal{H}_{\e,\eta}^{\mathrm{D}}}
\def\Hr{\mathcal{H}_{\e,\eta}^{\mathrm{R}}}
\def\Hod{\mathcal{H}_{0}^{\mathrm{D}}}
\def\Hn{\mathcal{H}_{\e,\eta}^{\mathrm{N}}}

\def\hd{\mathfrak{h}_{\e,\eta}^{\mathrm{D}}}
\def\hr{\mathfrak{h}_{\e,\eta}^{\mathrm{R}}}
\def\hod{\mathfrak{h}_{0}^{\mathrm{D}}}
\def\hor1{\mathfrak{h}_{0}^{\mathrm{R}%,\a
}}
\def\Hor1{\mathcal{H}_{0}^{\mathrm{R}%,\a
}}
\def\hn{\mathfrak{h}_{\e,\eta}^{\mathrm{N}}}
\def\Hon{\mathcal{H}_{0}^{\mathrm{N}}}
\def\hon{\mathfrak{h}_{0}^{\mathrm{N}}}
\DeclareMathOperator{\Dom}{\mathfrak{D}}

\newcounter{assumption}
\begin{document}

\title{\textbf{Uniform resolvent convergence for strip with fast oscillating boundary}}
\author{Denis Borisov\,$^a$\footnote{Corresponding author},  Giuseppe Cardone$^b$, Luisa Faella$^c$, Carmen Perugia$^d$}

\date{\small
\begin{center}
\begin{quote}
\begin{enumerate}
{\it
\item[$a)$]
Institute of Mathematics of Ufa Scientific Center of RAS, Chernyshevskogo, str. 112, 450008, Ufa, Russian Federation \&
Bashkir State Pedagogical University,
October St.~3a, 450000, Ufa,
Russian Federation; \\ E-mail: \texttt{borisovdi@yandex.ru}
\item[$b)$]
University of Sannio,
Department of Engineering, Corso Garibaldi,
107, 82100 Benevento, Italy; E-mail: \texttt{giuseppe.cardone@unisannio.it}
\item[$c)$] University of Cassino and Lazio Meridionale, DIEI ``M.Scarano'', Via G. di Biasio, 43
03043 Cassino (FR); E-mail: \texttt{l.faella@unicas.it}
\item[$d)$] University of Sannio, DST,
Via Port'Arsa 11, 82100 Benevento, Italy; E-mail: \texttt{ cperugia@unisannio.it}
}
\end{enumerate}
\end{quote}
\end{center}}

\maketitle

\begin{abstract}
In a planar infinite strip with a fast oscillating boundary we consider an elliptic operator assuming that both the period and the amplitude of the oscillations are small. On the oscillating boundary we impose Dirichlet, Neumann or Robin boundary condition. In all cases we describe the homogenized operator, establish the uniform resolvent convergence of the perturbed resolvent to the homogenized one, and prove the estimates for the rate of convergence. These results are obtained as the order of the amplitude of the oscillations is less, equal or greater than that of the period. It is shown that under the homogenization the type of the boundary condition can change.
\end{abstract}

\begin{center}
\begin{quote}
\end{quote}
\end{center}

%%-----------------------------
%%      the top matter
%%-----------------------------

%%-----------------------------
\section*{Introduction}

There are many papers devoted to the homogenization of boundary value problems in  domains with fast oscillating boundary. The simplest example of such boundary is given by the graph of the function $x_2=\eta(\e)b(x_1\e^{-1})$, where $\e$ is a small positive parameter, $\eta(\e)$ is a positive function tending to zero as $\e\to+0$, and $b$ is a smooth periodic function. The parameter $\e$ describes the period of the boundary oscillations while $\eta(\e)$  is their amplitude.

Most of the papers on such topic are devoted to the case of bounded domains with fast oscillating boundary. Not trying to cite all papers in this field, we just mention
\cite[Ch. I\!I\!I, Sec. 4]{OIS}, \cite{DD}--%, \cite{FHL},
%\cite{FH}, \cite{ACG}, \cite{CC},  \cite{ABCP}, \cite{ACG2}, %\cite{AB1}, \cite{AB2}, \cite{AB3}, \cite{AC}, \cite{BPC}, %\cite{CFP}, \cite{KN}, %\cite{GR},
%\cite{ChCh}, \cite{ChCh1}, \cite{MP}, \cite{Me}, \cite{Mi},
% \cite{Na3}, \cite{Na2},
\cite{Na4}, see also the references therein. Main results concerned  the identification of the homogenized problems and proving the  convergence theorems for the solutions. The homogenized (limiting) problems were the boundary value problems for the same equations in the same domains but with the mollified boundary instead of the oscillating one. The type of the condition on the mollified boundary depended on the original boundary condition and the geometry of the oscillations. If the amplitude of the oscillations is of the same order as the period (i.e., in above example $\eta\sim\e$),  the homogenized boundary condition is of the same type as the original condition on the oscillating boundary. In the case of Robin condition the homogenization gives rise to an additional term in the coefficient in the homogenized boundary condition; this term reflected the local geometry of the boundary oscillations. If the period of the boundary oscillations is smaller (in order) than the amplitude, the boundary is highly oscillating. To authors knowledge, such case was considered in \cite{CFP}, \cite{KN}, and \cite{MP}. In \cite{KN} the model was the spectral problem for the biharmonic operator with Dirichlet condition, while in \cite{CFP} the Robin problem for Poisson equation was studied. In the former case in particular it was shown that the homogenized boundary condition was the Dirichlet one while in the latter the authors discovered that in the case of highly oscillating boundary the homogenized boundary condition is also the Dirichlet while the perturbed problem involved the Robin condition. In \cite{MP} the Navier-Stokes system was considered and the boundary condition on the oscillating boundary was the Robin one with the small parameter involved in the coefficients. The homogenized problems were found and the weak convergence was established.

In \cite{AB1}--%, \cite{AB2}, \cite{AB3},
\cite{AC} a boundary value problem for the semilinear elliptic equation with non-linear boundary conditions of Robin type in a bounded domain was considered. The domain was assumed to be perturbed and the only assumption was that the perturbed domain converged to a certain limiting one in the sense of Hausdorff and the same was valid for the boundaries. This setting includes also the case of fast oscillating boundary. The main result was the proof of the convergence of the perturbed solution to the limiting one in $\H^1$-norm and similar statements for the spectra.

It should be said that there are also many papers devoted to the problems in the domains with the oscillating boundaries when the period of the oscillations is small and the amplitude is finite. Since in our case the amplitude is small, it is quite a different problem. This is why here we do not dwell on the problems with finite amplitude.

Most of the results on the convergence of the solutions were established in the sense of the weak or strong resolvent convergence, and the resolvents were also treated in various possible norms. In some cases the estimates for the convergence rate were proven. It was also shown that constructing the next terms of the asymptotics for the perturbed solutions one get the estimates for the convergence rate or  improves it \cite{FHL}, \cite{FH}, \cite{BPC}, \cite{KN}, \cite{ChCh}, \cite{Mi}, \cite{Na3}, \cite{Na4}.  In some cases   complete asymptotic expansions were constructed \cite{ACG},  \cite{ACG2}, \cite{Na2}, \cite{GR}.

One more type of the established results is the uniform resolvent convergence for the problems. Such convergence was established just for few models, see \cite[Ch. I\!I\!I, Sec. 4]{OIS}, \cite{Na4}. The estimates for the rates of convergence were also established. In both papers the amplitude and the period of oscillations were of the same order. At the same time, the uniform resolvent convergence for the models considered in the homogenization theory is a quite strong results. Moreover, recently the series of papers by M.Sh.~Birman, T.A.~Suslina and V.V.~Zhikov, S.E.~Pastukhova have stimulated the interest to this aspect, see  \cite{CBS}--%, \cite{CPZh}, \cite{PT}, \cite{Pas}, \cite{Bir1}, \cite{BS4}, \cite{BS2}, \cite{BS5},
%\cite{VS}, \cite{Zh3}, \cite{Zh6}, \cite{Zh4},
\cite{Zh5}, the references therein and further papers by these authors. It was shown that the uniform resolvent convergence holds true for the elliptic operators with fast oscillating coefficients and the estimates for the rates of convergence were obtained. There are also same results for some problems in bounded domains, see \cite{Zh4}. Similar results but for the boundary homogenization were established in \cite{BBC2}--%, %\cite{BBC}, \cite{BBC3}, \cite{BBC4},
\cite{BC}. Here the Laplacian in a planar straight infinite strip with frequently alternating boundary conditions was considered. Such boundary conditions were imposed by partitioning the boundary into small segments where Dirichlet and Robin conditions were imposed in turns. The homogenized problem involves one of the classical boundary conditions instead of the alternating ones. For all possible homogenized problems the uniform resolvent and the estimates for the rates of convergence were proven and the asymptotics for the spectra were constructed.

In the present paper we also consider the boundary homogenization for the elliptic operators in unbounded domains but the perturbation is a fast oscillating boundary. As the domain we choose a planar straight infinite strip with a periodic fast oscillating boundary; the operator is a general self-adjoint second order elliptic operator. The operator is regarded as an unbounded one in an appropriate $L_2$ space. On the oscillating boundary we impose Dirichlet, Neumann, or Robin condition. Apart from a mathematical interest to this problem, as a physical motivation we can mention a model of a planar quantum or acoustic waveguide with a fast oscillating boundary.

Our main result is the form of the homogenized operator and the uniform resolvent convergence of the perturbed operator to the homogenized one. This convergence is established in the sense of the norm of the operator acting from $L_2$ into $\H^1$. The estimates for the rate of convergence are provided. Most of the estimates are sharp. We show that in the case of the Dirichlet or Neumann condition on the oscillating boundary the homogenized problem involves the same condition on the mollified boundary no matter how the period and amplitude of the oscillations behave. Provided the amplitude is not greater than the period (in order), the Robin condition on the oscillating boundary leads us to a similar condition but with an additional term in the coefficient. If the amplitude is greater than the period, the homogenization transforms the Robin condition into the Dirichlet one. The last result is in a good accordance with a similar case  treated  in \cite{CFP}. The difference is that in \cite{CFP} the strong resolvent convergence was proven provided the coefficient in the Robin condition is positive, while we succeeded to prove the uniform resolvent convergence provided the coefficient is either positive or non-negative and vanishing on the set of zero measure.

\section{The problem and the main results}

Let $x=(x_1,x_2)$ be the Cartesian coordinates in $\mathds{R}^2$,
$\e$ be a small positive parameter, $\eta=\eta(\e)$ be a non-negative function uniformly bounded for sufficiently small $\e$, $b=b(t)$ be a
non-negative $1$-periodic function belonging to $C^2(\mathds{R})$.
We define two domains, cf. fig.~\ref{pic1},
\begin{equation*}
\Om_0:=\{x: 0<x_2<d\},\quad \Om_\e:=\{x:
\eta(\e)b(x_1\e^{-1})<x_2<d\},
\end{equation*}
where $d>0$ is a constant, and its boundaries are indicated as
\begin{equation*}
\G:=\{x: x_2=d\},\quad \G_0:=\{x: x_2=0\},\quad \G_\e:=\{x:
x_2=\eta(\e) b(x_1\e^{-1})\}.
\end{equation*}
By $A_{ij}=A_{ij}(x)$, $A_j=A_j(x)$, $A_0=A_0(x)$, $i,j=1,2$, we denote the functions defined on $\Om_0$ and satisfying the
belongings $A_{ij}\in\Hinf^2(\Om_0)$, $A_j\in\Hinf^1(\Om_0)$,
$A_0\in L_\infty(\Om_0)$. Functions $A_{ij}$, $A_j$ are assumed to be complex-valued, while $A_0$ is real-valued. In addition, functions $A_{ij}$ satisfy the ellipticity condition
\begin{equation}\label{2.0}
A_{ij}=\overline{A}_{ji},\quad \sum\limits_{i,j=1}^2 A_{ij}z_i
\overline{z}_j\geqslant c_0(|z_1|^2+|z_2|^2),\quad x\in\Om_0,\quad
z_j\in\mathds{C}.
\end{equation}
By $a=a(x)$ we denote a real function defined on $\{x:  0<x_2<\d\}$ for some small fixed $\d$, and it is supposed that $a\in\Hinf^1(\{x: 0<x_2<d\})$.

The main object of our study is the operator
\begin{equation}\label{2.1}
-\sum\limits_{i,j=1}^{2}\frac{\p}{\p x_j} A_{ij} \frac{\p}{\p x_i}+\sum\limits_{j=1}^{2} A_j\frac{\p}{\p x_j}-\frac{\p}{\p x_j} \overline{A_j} + A_0\quad \text{in}\quad L_2(\Om_\e)
\end{equation}
subject to  Dirichlet condition on $\G$. On the other
boundary we choose either Dirichlet condition
\begin{equation*}%\l%abel{2.2}
u=0\quad\text{on}\quad\G_\e,
\end{equation*}
or Robin condition
\begin{equation*}%\l%abel{2.3}
\left(\frac{\p}{\p\nu^\e}+a\right)u=0\quad\text{on}\quad\G_\e,\quad
\frac{\p}{\p\nu^\e}=-\sum\limits_{i,j=1}^{2} A_{ij}\nu_j^\e \frac{\p}{\p x_i}-\sum\limits_{j=1}^{2} \overline{A}_j\nu_j^\e,
\end{equation*}
where $\nu^\e=(\nu^\e_1,\nu^\e_2)$ is the outward normal to $\G_\e$. In the case of Dirichlet condition on $\G_\e$ we denote this operator as $\Hd$, while for   Robin condition it is $\Hr$. The former includes also the case of Neumann condition since the function $a$ can be identically zero.

Rigorously we introduce $\Hd$ as the lower-semibounded self-adjoint operator in $L_2(\Om_\e)$ associated with the closed symmetric lower-semibounded sesquilinear form
%\begin{equation}\l%abel{2.4}
\begin{align*}
\hd(u,v):=&\sum\limits_{i,j=1}^{2} \left(A_{ij}\frac{\p u}{\p x_j},\frac{\p v}{\p x_i}\right)_{L_2(\Om_\e)}+ \sum\limits_{j=1}^{2} \left(A_j \frac{\p u}{\p x_j},v\right)_{L_2(\Om_\e)}
\\
&+ \sum\limits_{j=1}^{2} \left(u,A_j \frac{\p v}{\p x_j}\right)_{L_2(\Om_\e)} + (A_0 u,v)_{L_2(\Om_\e)}
\end{align*}
%\end{equation}
in $L_2(\Om_\e)$ with the domain $\Dom(\hd):=\Ho^1(\Om_\e,\p\Om_\e)$. Hereinafter $\Dom(\cdot)$ is the domain of a form or an operator, and $\Ho^j(\Om,S)$ denotes the Sobolev space consisting of the functions in $\H^j(\Om)$ with zero trace on a curve $S$ lying in a domain $\Om\subset\mathds{R}^2$. The operator $\Hr$ is introduced in the same way via the sesquilinear form
%\begin{equation}\l%abel{2.5}
\begin{align*}
\hr(u,v):=&\sum\limits_{i,j=1}^{2} \left(A_{ij}\frac{\p u}{\p x_j},\frac{\p v}{\p x_i}\right)_{L_2(\Om_\e)}+ \sum\limits_{j=1}^{2} \left(A_j \frac{\p u}{\p x_j},v\right)_{L_2(\Om_\e)}
\\
&+ \sum\limits_{j=1}^{2} \left(u,A_j \frac{\p v}{\p x_j}\right)_{L_2(\Om_\e)} + (A_0 u,v)_{L_2(\Om_\e)} + (au,v)_{L_2(\G_\e)}
\end{align*}
%\end{equation}
with the domain $\Dom(\hr):=\Ho^1(\Om_\e,\G)$.

\begin{figure}\label{pic1}
\begin{center}
\includegraphics[scale=0.65]{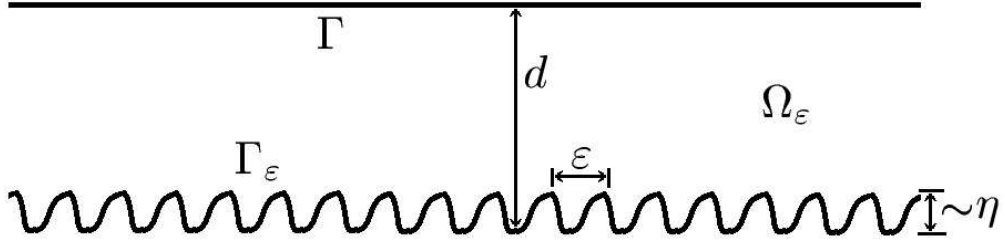}

\caption{Domain with oscillating boundary}
\end{center}
\end{figure}

The main aim of the paper is to study the asymptotic behavior of the resolvents of $\Hd$ and $\Hr$ as $\e\to+0$. To formulate the main results we first introduce some additional operators.

By $\Hod$ we denote operator (\ref{2.1}) in $L_2(\Om_0)$ subject to  Dirichlet condition. We introduce it by analogue with $\Hd$ as associated with the form
\begin{equation}\label{2.7}
\begin{aligned}
\hod(u,v):=&\sum\limits_{i,j=1}^{2} \left(A_{ij}\frac{\p u}{\p x_j},\frac{\p v}{\p x_i}\right)_{L_2(\Om_0)}+ \sum\limits_{j=1}^{2} \left(A_j \frac{\p u}{\p x_j},v\right)_{L_2(\Om_0)}
\\
&+\sum\limits_{j=1}^{2} \left(u,A_j \frac{\p v}{\p x_j}\right)_{L_2(\Om_0)} + (A_0 u,v)_{L_2(\Om_0)}
\end{aligned}
\end{equation}
in $L_2(\Om_0)$ with the domain $\Dom(\hod):=\Ho^1(\Om_0,\p\Om_0)$. The domain of  operator $\Hod$ is $\Ho^2(\Om_0,\p\Om_0)$ that can be shown by analogy with \cite[Ch. I\!I\!I, Sec. 7,8]{LU}, \cite[Lm. 2.2]{B-AA-08}.

Our first main result describes the uniform resolvent convergence for $\Hd$.

\begin{theorem}\label{th2.1}
Let $f\in L_2(\Om_0)$. For sufficiently small $\e$ the estimate
\begin{equation*}
\|(\Hd-\iu)^{-1}f-(\Hod-\iu)^{-1}f\|_{\H^1(\Om_\e)}\leqslant C\eta^{1/2} \|f\|_{L_2(\Om_0)}
\end{equation*}
holds true, where $C$ is a constant independent of $\e$ and $f$.
\end{theorem}

The next four theorems describe the resolvent convergence for operator $\Hr$. Given $a_0\in\Hinf^1(\G_0)$, let $\Hor1$ be the self-adjoint operator  in $L_2(\Om_0)$ associated with the lower-semibounded sesquilinear symmetric form
%\begin{equation}\l%abel{2.10}
\begin{align*}
\hor1(u,v):=&\sum\limits_{i,j=1}^{2} \left(A_{ij}\frac{\p u}{\p x_j},\frac{\p v}{\p x_i}\right)_{L_2(\Om_0)}+ \sum\limits_{j=1}^{2} \left(A_j \frac{\p u}{\p x_j},v\right)_{L_2(\Om_0)}
\\
&+ \sum\limits_{j=1}^{2} \left(u,A_j \frac{\p v}{\p x_j}\right)_{L_2(\Om_0)} + (A_0 u,v)_{L_2(\Om_0)} + (a_0u,v)_{L_2(\G_0)}
\end{align*}
%\end{equation}
with the domain $\Dom(\hor1):=\Ho^1(\Om_0,\G)$. It can be shown by analogy with \cite[Ch. I\!I\!I, Sec. 7,8]{LU}, \cite[Lm. 2.2]{B-AA-08} that the domain of $\Hor1$ consists of the functions $u \in \Ho^2(\Om_0,\G)$ satisfying Robin condition
\begin{equation}\label{2.3a}
\left(\frac{\p\hphantom{\nu^0}}{\p\nu^0}+a_0\right)u=0 \quad\text{on}\quad\G_0,\quad
\frac{\p\hphantom{\nu^0}}{\p\nu^0}:=-\sum\limits_{i=1}^{2} A_{i2} \frac{\p\hphantom{x_i}}{\p x_i}-\overline{A}_2.
\end{equation}

First we consider the particular case of Neumann condition on $\G_\e$, i.e., $a=0$. Operator $\Hr$ and  associated quadratic form $\hr$ are re-denoted in this case by $\Hn$ and $\hn$. By $\Hon$ we denote the self-adjoint lower-semibounded operator in $L_2(\Om_0)$ associated with the sesquilinear form $\hon$ which is $\hor1$ taken for $a_0\equiv 0$. Its domain is the set of the functions in $\Ho^2(\Om_0,\G)$ satisfying  boundary condition (\ref{2.3a}) with $a_0=0$. The resolvent convergence in this case is given in
\begin{theorem}\label{th2.3}
Let $f\in L_2(\Om_\e)$.
Then
for sufficiently small $\e$ the estimate
\begin{equation*}
\|(\Hn-\iu)^{-1}f-(\Hon-\iu)^{-1}f\|_{\H^1(\Om_\e)}\leqslant
C \eta^{1/2}  \|f\|_{L_2(\Om_0)}
\end{equation*}
holds true, where $C$ is a constant independent of $\e$ and $f$.
\end{theorem}

Assume now $a\not\equiv 0$. Here we consider separately two cases,
\begin{align}
&\e^{-1}\eta(\e)\to \a=\mathrm{const}\geqslant 0, \quad\e\to+0,\label{2.8}
\\
&\e^{-1}\eta(\e)\to+\infty, \hphantom{\mathrm{cons}\geqslant M} \quad  \e\to+0.\label{2.9}
\end{align}

The first assumption means that the amplitude of the oscillation of curve $\G_\e$ is of the same order (or smaller) as the period. The other assumption corresponds to the case when the amplitude is much greater than the period. In what follows the first case is referred to as a relatively slow oscillating boundary $\G_\e$ while the other describes relatively high oscillating boundary $\G_\e$.

We begin with the slowly oscillating boundary. We denote
\begin{equation}\label{2.10a}
a_0(x_1):=a(x_1,0)\int\limits_0^1\sqrt{1+\a^2\big(b'(t)\big)^2}  \di t.
\end{equation}

\begin{theorem}\label{th2.2}
Suppose (\ref{2.8}) and let $f\in L_2(\Om_\e)$. Then
for sufficiently small $\e$ the estimate
\begin{equation*}
\|(\Hr-\iu)^{-1}f-(\Hor1-\iu)^{-1}f\|_{\H^1(\Om_\e)}\leqslant
C(\eta^{1/2}(\e)+|\e^{-2}\eta^2(\e)-\a^2|)\|f\|_{L_2(\Om_0)}
\end{equation*}
holds true, where  function $a_0$ in (\ref{2.3a}) is defined in (\ref{2.10a}), and $C$ is a constant independent of $\e$ and $f$.
\end{theorem}

We proceed to the case of the highly oscillating boundary $\G_\e$. Here the homogenized operator happens to be quite sensitive to the sign of $a$ and zero level set of this function. In the paper we describe the resolvent convergence as $a$ is non-negative. We first suppose that $a$ is bounded from below by a positive constant. Surprisingly, but here the homogenized operator has the Dirichlet condition on $\G_0$ as in Theorem~\ref{th2.1}.

\begin{theorem}\label{th2.4} Suppose (\ref{2.9}),
\begin{equation}\label{2.11}
a(x)\geqslant c_1>0,\quad c_1=\mathrm{const},
\end{equation}
and that the function $b$ is not identically constant.
Let $f\in L_2(\Om_0)$.
Then
for sufficiently small $\e$ the estimate
\begin{equation}\label{1.1}
\|(\Hr-\iu)^{-1}f-(\Hod-\iu)^{-1}f\|_{\H^1(\Om_\e)}\leqslant
C \big(\eta^{1/2}+\e^{1/2}\eta^{-1/2}\big) \|f\|_{L_2(\Om_0)}
\end{equation}
holds true, where $C$ is a constant independent of $\e$ and $f$.
\end{theorem}

In the next theorem we still suppose that $a$ is non-negative but can have zeroes. An essential assumption is that zero level set of $a$ is of zero measure. We let $b_*:=\max\limits_{[0,1]} b$.

\begin{theorem}\label{th2.5}
Suppose (\ref{2.9}),
\begin{equation}\label{2.12}
a\geqslant 0,
\end{equation}
and that the function $b$ is not identically constant. Assume also that for all sufficiently small $\d$ the set $\{x: a(x)\leqslant \d, \ 0<x_2<(b_*+1)\eta\}$ is contained in an at most  countable union of the rectangles $\{x: |x_1-X_n|<\mu(\d),\ 0<x_2<(b_*+1)\eta\}$, where $\mu(\d)$ is a some nonnegative function such that $\mu(\d)\to+0$ as $\d\to+0$, and numbers $X_n$, $n\in \mathds{Z}$, are independent of $\d$, are taken in the ascending order, and satisfy uniform in $n$ and $m$ estimate
\begin{equation}\label{2.13}
|X_n-X_m|\geqslant c>0,\quad n\not=m.
\end{equation}
Let $f\in L_2(\Om_0)$.
Then for sufficiently small $\e$ the estimate
\begin{equation}
\begin{aligned}
\|(\Hr-\iu)^{-1}f-&(\Hod-\iu)^{-1}f\|_{\H^1(\Om_\e)}
\\
&\leqslant
C \big(\eta^{1/2}+ \e^{1/2}\eta^{-1/2}\d^{-1/2}+
\mu^{1/2}(\d)|\ln\mu(\d)|^{1/2}\big) \|f\|_{L_2(\Om_0)}
\end{aligned}\label{1.2}
\end{equation}
holds true, where $C$ is a constant independent of $\e$ and $f$, and $\d=\d(\e)$ is any function tending to zero as $\e\to+0$.
\end{theorem}

Let us discuss the main results. We first observe that under the hypotheses of all theorems we have the corresponding spectral convergence, namely, the convergence of the spectrum and the associated spectral projectors -- see, for instance, \cite[Thms. V\!I\!I\!I.23, V\!I\!I\!I.24]{RS1}. We also stress that in all Theorems~\ref{th2.1}-\ref{th2.5} the resolvent convergence is established in the sense of the uniform norm of bounded operator acting from  $L_2(\Om_0)$ into $\H^1(\Om_\e)$.

In the case of the Dirichlet condition on $\G_\e$ the homogenized operator has the same condition on $\G_0$ no matter how the boundary $\G_\e$ oscillates, slowly or highly. The estimate for the rate of convergence is also universal being $\Odr(\eta^{1/2})$. Despite here we consider a periodically oscillating boundary, in the proof of Theorem~\ref{th2.1} this fact is not used. This is why its statement is valid also for a periodically oscillating boundary described by the equation $x_2=\eta b(x_1,\e)$, where $b$ is an arbitrary function bounded uniformly in $\e$ and such that $b(\cdot,\e)\in C(\mathds{R})$. The estimate is Theorem~\ref{th2.1} is sharp, see the discussion in the end of Sec.~2.

Similar situation happens if we have Neumann condition on $\G_\e$. Here Theorem~\ref{th2.3} says that the homogenized operator is subject to Neumann condition on $\G_0$ and the rate of the uniform resolvent convergence is the same as in Theorem~\ref{th2.1}, namely, $\Odr(\e^{1/2})$. This estimate is again sharp, as the  example in the end of Sec. 3 shows.

Once we have Robin condition on $\G_\e$, the situation is completely different. If the boundary oscillates slowly, the homogenized operator still has Robin condition on $\G_0$, but the coefficient depends on the geometry of the original oscillations, cf. (\ref{2.10a}). The estimate for the rate of the resolvent convergence in this case involves additional term in comparison with the Dirichlet or Neumann case, cf. Theorem~\ref{th2.2}. The estimate in this theorem is again sharp, see the  example in the end of Sec. 3.

As  boundary $\G_\e$ oscillates relatively high, the resolvent convergence  changes dramatically.
If coefficient $a$ is strictly positive,  the homogenized operator has the Dirichlet condition on $\G_0$. A new term, $\e^{1/2}\eta^{-1/2}$, appears in the estimate for the rate of the uniform resolvent convergence, cf. Theorem~\ref{th2.4}. We are able to prove that this term is sharp, see the discussion in the end of Sec.~4.

Provided function $a$ is non-negative and vanishes only on a set of zero measure, the homogenized operator still has Dirichlet condition on $\G_0$, but the estimate for the rate of the uniform resolvent convergence becomes worse. Namely, the behavior of $a$ in a vicinity of its zeroes becomes important. It is reflected by functions $\mu(\d)$ and $\d$ in (\ref{1.2}). The latter should be chosen so that $\d\to+0$, $\e^{1/2}\eta^{-1/2}\d^{-1/2}\to+0$, $\e\to+0$, that is always possible. The optimal choice of $\d$ is so that
\begin{align}
&\mu^{1/2}(\d)|\ln\mu(\d)|^{1/2}\sim \e^{1/2}\eta^{-1/2} \d^{-1/2},\nonumber
\\
&\d \mu(\d) |\ln\mu(\d)| \sim \e \eta^{-1}. \label{2.14}
\end{align}
As we see, the choice of $\d$ depends on a particular structure of $\mu(\d)$.
The most typical case is $\mu(\d)\sim \d^{1/2}$, i.e., the function $a$ vanishes by the quadratic law in a vicinity of its zeroes. In this case  condition (\ref{2.14}) becomes
\begin{equation*}
\d^{3/2} |\ln\d|\sim\e\eta^{-1}
\end{equation*}
that implies
\begin{equation*}
\d\sim \e^{2/3}\eta^{-2/3} |\ln\e \eta^{-1}|^{-2/3}.
\end{equation*}
Then the estimate for the resolvent convergence in Theorem~\ref{th2.5} is of order $\Odr\big( (\eta^{1/2} +\e^{1/6}\eta^{-1/6} |\ln\e\eta^{-1}|^{1/3}
\big)$.

We are not able to prove the sharpness of estimate (\ref{1.2}), but in the end of Sec.~4 we provide some arguments showing that estimate (\ref{1.2})  is rather close to be optimal.

In conclusion we discuss the case of Robin condition on highly oscillating $\G_\e$ when the coefficient $a$ does not satisfy the hypotheses of Theorems~\ref{th2.4},~\ref{th2.5}. If it is still non-negative but vanishes on a set of non-zero measure, and at the end-points of this set the vanishing happens with certain rate like in Theorem~\ref{th2.5}, we conjecture that the homogenized operator involves mixed Dirichlet and Neumann condition on $\G_0$. Namely, if $a(x_1,0)\equiv0$ on $\G_0^N$ and $a(x_1,0)>0$ on $\G_0^D$, $\G_0=\G_0^N\cup\G_0^D$,  it is natural to expect that the homogenized operator has Neumann condition on $\G_0^N$ and  Dirichlet one on $\G_0^D$. This conjecture can be regarded as the mixture of the statements of Theorems~\ref{th2.3}~and~\ref{th2.5}. The main difficulty of proving this conjecture is that the domain of such homogenized operator is no longer a subset of $\H^2(\Om_0)$ because of the mixed boundary conditions. At the same time, this fact was essentially used in all our proofs. Even a more complicated situation occurs once $a$ is negative or sign-indefinite. If $a$ is negative on a set of non-zero measure, it can be shown that the bottom of the spectrum of the perturbed operator goes to $-\infty$ as $\e\to+0$. In such case one should study the resolvent convergence near this bottom, i.e., for the spectral parameter tending to $-\infty$. This makes the issue quite troublesome. We stress that under the hypotheses of all Theorems~\ref{th2.1}-\ref{th2.5} the bottom of the spectrum is lower-semibounded uniformly in $\e$.

\section{Dirichlet condition}

In this section we study the resolvent convergence of the operator $\Hd$ and prove Theorem~\ref{th2.1}. We begin with auxiliary lemma.

\begin{lemma}\label{lm3.5} Suppose $u\in\Ho^2(\Om_0, \G_0)$, $v\in\Ho^1(\Om_\e,\G_\e)$. Then  the estimates
\begin{align*}
&|u(x)|^2\leqslant C x_2^2\|u(x_1,\cdot)\|_{\H^2(0,d)}^2, &&\text{for a.e.}\quad x_1\in\mathds{R},\quad x_2\in(0,d/2),
\\
&|\nabla u(x)|^2\leqslant C \|\nabla u(x_1,\cdot)\|_{\H^1(0,d)}^2,&&\text{for a.e.}\quad x_1\in\mathds{R},\quad x_2\in(0,d/2),
\\
&|v(x)|^2\leqslant C x_2 \|v(x_1,\cdot)\|_{\H^1(\eta b(x_1\e^{-1}),d)}^2, &&\text{for a.e.}\quad x_1\in\mathds{R},\quad x_2\in(\eta b(x_1\e^{-1}),d/2),
\end{align*}
hold true, where $C$ are constants independent of $x$, $\e$, $u$, and $v$.
\end{lemma}

\begin{proof}
Since $u\in\H^2(\Om_0)$, for almost all $x_1\in\mathds{R}$ we have $u(x_1,\cdot)\in\H^2(0,d)$. We represent the function $u$ as
\begin{equation*}
u(x_1,x_2)=\int\limits_{0}^{x_2}\frac{\p u}{\p x_2}(x_1,t)\di t,
\end{equation*}
and by Cauchy-Schwarz inequality we obtain
\begin{equation}\label{3.36}
|u(x_1,x_2)|^2\leqslant C x_2\int\limits_{0}^{x_2} \bigg|\frac{\p u}{\p x_2}(x_1,t)\bigg|^2\di t.
\end{equation}
Let $\chi_1=\chi_1(x_2)$ be an infinitely differentiable smooth function vanishing as $x_2>3d/4$ and equalling one as $x_2<d/2$. Then for $x_2\in[0,d/2]$ we have
\begin{equation*}
\frac{\p u}{\p x_2}(x_1,x_2)=\int\limits_{d}^{x_2} \left(\frac{\p}{\p x_2}\chi_1\frac{\p u}{\p x_2}\right)(x_1,t)\di t,
\end{equation*}
and thus
\begin{equation*}
\bigg|\frac{\p u}{\p x_2}(x_1,x_2)\bigg|^2\leqslant C\int\limits_{0}^{d} \left( \bigg|\frac{\p^2 u}{\p x_2^2}(x_1,t)
\bigg|^2+\bigg|\frac{\p u}{\p x_2}(x_1,t)
\bigg|^2 \right)\di t.
\end{equation*}
Substituting this inequality into (\ref{3.36}), we arrive at the first required estimate. To prove two others one should proceed as above starting with the representation
\begin{equation*}
v(x_1,x_2)=\int\limits_{0}^{x_2} \frac{\p v}{\p x_2}(x_1,t)\di t,
\end{equation*}
where $v$ is assumed to be extended by zero outside $\Om_\e$, and the representation
\begin{equation*}
\frac{\p u}{\p x_j}(x_1,x_2)=\int\limits_{d}^{x_2} \left(\frac{\p }{\p x_2}\chi_1\frac{\p u}{\p x_j}\right)(x_1,t)\di t.
\end{equation*}
\end{proof}

\begin{proof}[Proof of Theorem~\ref{th2.1}]
By $\chi_2=\chi_2(t)$ we denote an infinitely differentiable non-negative cut-off function with the values in $[0,1]$ vanishing as $t>1$ and being one as $t<0$. We also assume that the values of $\chi_2$ are in $[0,1]$. We choose a function $K$ as
\begin{equation}\label{3.33}
K(x_2,\eta):=\chi_2\left(\frac{x_2-b_*\eta}{\eta}\right).
\end{equation}
We observe that the function $(1-K)$ vanishes for $0<x_2<b_*\eta$ and is independent of $x_1$.

Given a function  $f\in L_2(\Om_0)$, we denote $u_\e:=(\Hd-\iu)^{-1}f$, $u_0:=(\Hod-\iu)^{-1}f$, $v_\e:=u_\e-(1-K)u_0$. In accordance with the definition of $u_\e$ and $u_0$, these functions satisfy the integral identities
\begin{equation}\label{3.3}
\hd(u_\e,\phi)+\iu(u_\e,\phi)_{L_2(\Om_\e)}=(f,\phi)_{L_2(\Om_\e)}
\end{equation}
for each $\phi\in\Ho^1(\Om_\e,\p\Om_\e)$, and
\begin{equation}\label{3.4}
\hod(u_0,\phi)+\iu(u_0,\phi)_{L_2(\Om_0)}=(f,\phi)_{L_2(\Om_0)}
\end{equation}
for each $\phi\in\Ho^1(\Om_0,\p\Om_0)$. It is clear that $(1-K)v_\e\in\Ho^1(\Om_\e,\p\Om_\e)$, $(1-K)u_0\in\Ho^1(\Om_\e,\p\Om_\e)$, and the extension of $v_\e$ by zero in $\Om_0\setminus\Om_\e$ belongs to $\Ho^1(\Om_0,\p\Om_0)$. Bearing these facts in mind, as the test function in (\ref{3.3}) we choose $\phi=v_\e$, and in (\ref{3.4}) we let $\phi=(1- K)v_\e$ assuming that $v_\e$ is extended by zero in $\Om_0\setminus\Om_\e$. It yields
\begin{align*}
&\hd(u_\e,v_\e)+\iu(u_\e,v_\e)=(f,v_\e)_{L_2(\Om_\e)},%\l%abel{3.5}
\\
&\hod(u_0,(1- K)v_\e)+\iu(u_0,(1- K)v_\e)_{L_2(\Om_\e)} =(f,(1- K)v_\e)_{L_2(\Om_\e)}.
%\l%abel{3.6}
\end{align*}
Employing (\ref{2.7}), we rewrite the term $\hod(u_0,(1-K)v_\e)$,
\begin{equation}
\begin{aligned}
\hod(u_0,&(1- K)v_\e)=\sum\limits_{i,j=1}^{2}\left((1-K)A_{ij}\frac{\p u_0}{\p x_j},\frac{\p v_\e}{\p x_i}\right)_{L_2(\Om_\e)}+ \sum\limits_{j=1}^{2} \left(A_j(1-K) \frac{\p u_0}{\p x_j},v_\e\right)_{L_2(\Om_\e)}
\\
&+  \sum\limits_{j=1}^{2} \left((1-K)u_0,A_j \frac{\p v_\e}{\p x_j}\right)_{L_2(\Om_\e)} + (A_0 (1-K)u_0,v_\e)_{L_2(\Om_\e)}
\\
&-\sum\limits_{j=1}^{2} \left(A_{2j}\frac{\p u_0}{\p x_j}, v_\e\frac{\p  K}{\p x_2}\right)_{L_2(\Om_\e)}- \left(u_0,A_2 \frac{\p  K}{\p x_2}v_\e\right)_{L_2(\Om_\e)}
\\
=&\hd((1-K)u_0,v_\e)+  \left(A_2 u_0,v_\e \frac{\p K}{\p x_2}\right)_{L_2(\Om_\e)}- \left(u_0,A_2 \frac{\p  K}{\p x_2}v_\e\right)_{L_2(\Om_\e)}
\\
&-\sum\limits_{j=1}^{2} \left(A_{2j}\frac{\p u_0}{\p x_j}, v_\e\frac{\p  K}{\p x_2}\right)_{L_2(\Om_\e)}+\sum\limits_{i =1}^{2}\left(A_{i2}u_0\frac{\p K}{\p x_2},\frac{\p v_\e}{\p x_i}\right)_{L_2(\Om_\e)}.
\end{aligned}\label{3.6a}
\end{equation}
It implies
\begin{equation}\label{3.2a}
\begin{aligned}
\hd(v_\e,v_\e)&+\iu\|v_\e\|_{L_2(\Om_\e)}^2=(f,Kv_\e)_{L_2(\Om_\e)} \\
&+  \left(A_2 u_0,v_\e \frac{\p K}{\p x_2}\right)_{L_2(\Om_\e)}- \left(u_0,A_2 \frac{\p  K}{\p x_2}v_\e\right)_{L_2(\Om_\e)}
\\
&-\sum\limits_{ j=1}^{2} \left(A_{2j}\frac{\p u_0}{\p x_j}, v_\e\frac{\p  K}{\p x_2}\right)_{L_2(\Om_\e)} +\sum\limits_{i =1}^{2}\left(A_{i2}u_0\frac{\p K}{\p x_2},\frac{\p v_\e}{\p x_i}\right)_{L_2(\Om_\e)}.
\end{aligned}
\end{equation}

The main idea of the proof  is to estimate the right hand side of (\ref{3.2a}) with the introduced function $K$ and get in this way an estimate for $v_\e$.

We first observe obvious inequalities
\begin{equation*}%\l%abel{3.34}
\|u_0\|_{\H^1(\Om_0)}\leqslant C\|f\|_{L_2(\Om_0)},\quad \|u_\e\|_{\H^1(\Om_0)}\leqslant C\|f\|_{L_2(\Om_0)}.
\end{equation*}
Here and till the end of the section by $C$ we denote inessential constants independent of $\e$, $x$, and $f$.  Proceeding as in \cite[Ch. I\!I\!I, Sec. 7,8]{LU} (see also \cite[Lm. 2.2]{B-AA-08}), one can also check that
\begin{equation}\label{3.35}
\|u_0\|_{\H^2(\Om_0)}\leqslant C\|f\|_{L_2(\Om_0)}.
\end{equation}

Denote $\Om^\eta:=\Om_\e\cap\{x: 0<x_2<(b_*+1)\eta\}$. Since the function $K$ vanishes outside $\Om^\eta$ and $|\nabla K|\leqslant C\eta^{-1}$, $0\leqslant K\leqslant 1$, it is easy to see that
\begin{equation}\label{3.11b}
\begin{aligned}
&\Bigg|(f,Kv_\e)_{L_2(\Om_\e)} + \left(A_2 u_0,v_\e \frac{\p K}{\p x_2}\right)_{L_2(\Om_\e)}-  \left(u_0,A_2 \frac{\p  K}{\p x_2}v_\e\right)_{L_2(\Om_\e)}
\\
&-\sum\limits_{ j=1}^{2} \left(A_{2j}\frac{\p u_0}{\p x_j}, v_\e\frac{\p  K}{\p x_2}\right)_{L_2(\Om_\e)} +\sum\limits_{i =1}^{2}\left(A_{i2}u_0\frac{\p K}{\p x_2},\frac{\p v_\e}{\p x_i}\right)_{L_2(\Om_\e)}
\Bigg|
\\
&\leqslant C \Big(\|f\|_{L_2(\Om_\e)}\|v_\e\|_{L_2(\Om^\eta)}
+\eta^{-1} \|u_0\|_{\H^1(\Om^\eta)}\|v_\e\|_{L_2(\Om^\eta)}
\\
&\hphantom{\leqslant C \Big(}   + \eta^{-1}\|u_0\|_{L_2(\Om^\eta)}\|\nabla v_\e\|_{L_2(\Om^\e)}
\Big).
\end{aligned}
\end{equation}
We estimate the terms in the right hand side by applying Lemma~\ref{lm3.5},
\begin{equation}\label{3.11a}
\begin{aligned}
&\|v_\e\|_{L_2(\Om^\eta)}^2= \int\limits_{\mathds{R}} \|v_\e(x_1,\cdot)\|_{L_2(\eta b(x\e^{-1}),(b_*+1)\eta)}^2\di x_1
\\
&\hphantom{\|v_\e\|_{L_2(\Om^\eta)}^2=}\leqslant C\|v_\e\|_{\H^1(\Om_\e)}^2\int\limits_{0}^{(b_*+1)\eta} x_2\di x_2\leqslant C\eta^2\|v_\e\|_{\H^1(\Om_\e)}^2,
\\
&\|u_0\|_{L_2(\Om^\eta)}^2\leqslant C\|u_0\|_{\H^2(\Om_\e)}^2\int\limits_{0}^{(b_*+1)\eta} x_2^2\di x_2\leqslant C\eta^3\|u_0\|_{\H^2(\Om_\e)}^2,
\\
&\|\nabla u_0\|_{L_2(\Om^\eta)}^2 \leqslant C\|u_0\|_{\H^2(\Om_\e)}^2\int\limits_{0}^{(b_*+1)\eta}\di x_2\leqslant C\eta\|u_0\|_{\H^2(\Om_\e)}^2.
\end{aligned}
\end{equation}
We substitute the obtained estimates and (\ref{3.35}) into (\ref{3.11b}),
\begin{align*}
&\Bigg|(f,Kv_\e)_{L_2(\Om_\e)} + \left(A_2 u_0,v_\e \frac{\p K}{\p x_2}\right)_{L_2(\Om_\e)}-  \left(u_0,A_2 \frac{\p  K}{\p x_2}v_\e\right)_{L_2(\Om_\e)}
\\
&-\sum\limits_{ j=1}^{2} \left(A_{2j}\frac{\p u_0}{\p x_j}, v_\e\frac{\p  K}{\p x_2}\right)_{L_2(\Om_\e)} +\sum\limits_{i =1}^{2}\left(A_{i2}u_0\frac{\p K}{\p x_2},\frac{\p v_\e}{\p x_i}\right)_{L_2(\Om_\e)}
\Bigg|
\\
&\leqslant C \eta^{1/2} \|f\|_{L_2(\Om_\e)}\|v_\e\|_{\H^1(\Om_\e)}.
\end{align*}
We take the real and imaginary parts of the right hand side in (\ref{3.2a}) and employ then the last obtained estimate and (\ref{3.35}). It leads us to the final estimate for $v_\e$,
\begin{equation*}
\|v_\e\|_{\H^1(\Om_\e)}\leqslant C\eta^{1/2}\|f\|_{L_2(\Om_\e)}.
\end{equation*}
Using (\ref{3.35}), by analogy with (\ref{3.11a}) one can check easily that
\begin{equation}\label{3.13}
\|K u_0\|_{\H^1(\Om_0)}\leqslant C\eta^{1/2}\|u_0\|_{\H^2(\Om_0)}\leqslant C\eta^{1/2}\|f\|_{L_2(\Om_0)}.
\end{equation}
The statement of Theorem~\ref{th2.1} follows from two last estimates and the definition of $v_\e$. The proof is complete.
\end{proof}

In conclusion let us discuss the optimality of the estimate in Theorem~\ref{th2.1}. Suppose for simplicity that the differential expression for $\Hd$ is just a negative Laplacian, $\eta(\e)\equiv\e$, $b\in C^\infty[0,1]$, $f\to C_0^\infty(\Om_0)$, $u_0\in C^\infty(\overline{\Om_0})$ and $u_0$ vanishes for sufficiently large $|x_2|$. Under such assumptions by the method of matching of asymptotic expansions \cite{Il} and the multiscale method \cite{MM} one can construct the asymptotic expansion for $u_\e$; for a similar spectral problem see \cite{ACG}, \cite{ACG2}. The asymptotics holds in $\H^1(\Om_\e)$-norm and for $\e b(x_1\e^{-1})<x_2<\e^{1/2}$ it reads as
\begin{equation*}
u_\e(x)=\frac{\p u_0}{\p x_2}(x_1,0) \big(x_2+\e Y(x\e^{-1})\big)+\Odr(\e),
\end{equation*}
where $Y=Y(\xi)$, $\xi=(\xi_1,\xi_2)$ is the $1$-periodic solution to the boundary value problem
\begin{align*}
&\D_\xi Y=0,\qquad\xi_1\in(0,1),\quad \xi_2>b(\xi_1),
\\
&Y=-\xi_2,\hspace{1 pt} \qquad \xi_1\in(0,1),\quad \xi_2=b(\xi_1),
\end{align*}
decaying exponentially as $\xi_2\to+\infty$. Expanding then $u_0$ into Taylor series as $x_2\to+0$, one can check easily that
\begin{align*}
\|u_\e&-u_0\|_{\H^1(\Om_\e\cap\{x: x_2<\e^{1/2}\})}=\e\left\|\nabla_x \frac{\p u_0}{\p x_2}(x_1,0) Y(x\e^{-1})\right\|_{L_2(\Om_\e\cap\{x: x_2<\e^{1/2}\})}+\Odr(\e)
\\
&=\left\|\frac{\p u_0}{\p x_2}(x_1,0)\nabla_\xi Y(x\e^{-1})\right\|_{L_2(\Om_\e\cap\{x: x_2<\e^{1/2}\})} + \Odr(\e)
\\
& = \e^{1/2} \left(
\int\limits_{\mathds{R}} \di x_1 \left|\frac{\p u_0}{\p x_2}(x_1,0)\right|^2 \int\limits_{b(x_1\e^{-1})}^{\e^{-1/2}}
\left|\nabla_\xi Y(x_1\e^{-1},\xi_2)\right|\di \xi_2
\right)^{1/2}+\Odr(\e)
\\
&=\e^{1/2} \left(
\int\limits_{\mathds{R}} \di x_1 \left|\frac{\p u_0}{\p x_2}(x_1,0)\right|^2 \int\limits_{b(x_1\e^{-1})}^{+\infty}
\left|\nabla_\xi Y(x_1\e^{-1},\xi_2)\right|\di \xi_2
\right)^{1/2}+\Odr(\e)
\end{align*}
and thus $\|u_\e-u_0\|_{\H^1(\Om_\e)}\geqslant C\e^{1/2}$. It proves the optimality of the estimate in Theorem~\ref{th2.1}.

\section{Robin condition on relatively slow oscillating boundary and Neumann condition}

In this section we study the resolvent convergence for  operators $\Hn$, $\Hr$ and prove Theorems~\ref{th2.3},~\ref{th2.2}. Throughout the section by $C$ we indicate various inessential constants independent of $\e$, $x$, and $f$.

We begin with two auxiliary lemmata. In these lemmata and their proofs constants $C$ are supposed to be independent of $\e$, $x$, and $u$.

The first lemma is an analogue of Lemma~\ref{lm3.5}.
\begin{lemma}\label{lm4.1}
For all $u\in\H^1(\Om_\e)$ and almost each $x_1\in\mathds{R}$, $x_2\in(\eta b(x_1\e^{-1}),d/2)$ the estimate
\begin{equation*}
|u(x)|\leqslant C\|u(x_1,\cdot)\|_{\H^1(0,\eta b(x_1\e^{-1}))},
\end{equation*}
holds true.
\end{lemma}

\noindent The proof of this lemma is similar to that of Lemma~\ref{lm3.5}. One just should employ the obvious identity
\begin{equation*}
u(x)=\int\limits_{d}^{x_2} \frac{\p \chi_1 u}{\p x_2}(x_1,t)\di t,
\end{equation*}
where $\chi_1$  was defined in the proof of Lemma~\ref{lm3.5}.

The next lemma gives an a priori estimate for the forms $\hn$, $\hr$.
\begin{lemma}\label{lm4.2}
For any $u\in\Ho^1(\Om_\e,\G)$ the estimate
\begin{equation}\label{3.4a}
\|u\|_{\H^1(\Om_\e)}^2\leqslant C\big(\hn(u,u)+\|u\|_{L_2(\Om_\e)}^2\big)
\end{equation}
holds true.

Suppose (\ref{2.8}). Then for any $u\in\Ho^1(\Om_\e,\G)$ the estimate
\begin{equation}\label{3.4b}
\|u\|_{\H^1(\Om_\e)}^2\leqslant C\big(\hr(u,u)+\|u\|_{L_2(\Om_\e)}^2\big)
\end{equation}
is valid.
\end{lemma}
\begin{proof}
It is clear that
\begin{equation}\label{4.12}
\left|
\sum\limits_{j=1}^{2} \left(A_j \frac{\p u}{\p x_j},u\right)_{L_2(\Om_\e)}
+ \sum\limits_{j=1}^{2} \left(u,A_j \frac{\p u}{\p x_j}\right)_{L_2(\Om_\e)}
\right|\leqslant \frac{c_0}{4}\|\nabla u\|_{L_2(\Om_\e)}^2+C\|u\|_{L_2(\Om_\e)}^2.
\end{equation}
This inequality and (\ref{2.0}) imply (\ref{3.4a}).

To prove (\ref{3.4b}), we just need to estimate the boundary integral over $\G_\e$ in the definition of $\hr$.
For $x\in\G_\e$ we have
\begin{equation*}%\l%abel{4.10}
|u(x)|^2\leqslant \int\limits_{\eta b(x_1\e^{-1})}^{d} \frac{\p |u|^2}{\p x_2}(x_1,t)\di t\leqslant \d\|\nabla u\|_{L_2(\Om_\e)}^2 +C(\d)\|u\|_{L_2(\Om_\e)}^2,
\end{equation*}
where the constant $\d$ can be chosen arbitrarily small. Hence, due to (\ref{2.8}), for an appropriate choice of $\d$
%\begin{equation}\l%abel{4.11}
\begin{align*}
|(a u,u)_{L_2(\G_\e)}|&=\Big|\int\limits_{\mathds{R}} a(x)|u(x)|^2\sqrt{1+\e^{-2}\eta^2 \big(b'(x_1\e^{-1})\big)^2}\di x_1
\Big|
\\
&\leqslant \frac{c_0}{4}\|\nabla u\|_{L_2(\Om_\e)}^2+C\|u\|_{L_2(\Om_\e)}^2.
\end{align*}
%\end{equation}
By this inequality, (\ref{4.12}), and (\ref{2.0}) we get the desired estimate.
\end{proof}

\begin{proof}[Proof of Theorem~\ref{th2.3}]
Denote $u_\e:=(\Hn-\iu)^{-1}f$, $u_0:=(\Hon-\iu)^{-1}f$, $v_\e:=u_\e-u_0$. The latter function solves the boundary value problem
\begin{align*}
&\left(-\sum\limits_{i,j=1}^{2}\frac{\p}{\p x_j} A_{ij} \frac{\p}{\p x_i}+\sum\limits_{j=1}^{2} A_j\frac{\p}{\p x_j}-\frac{\p}{\p x_j} \overline{A_j} + A_0-\iu\right) v_\e=0\quad\text{in}\quad\Om_\e,
\\
&\hphantom{\Bigg(-} v_\e=0\quad\text{on}\quad \G,\qquad \frac{\p v_\e}{\p\nu^\e}=-\frac{\p u_0}{\p \nu^\e}\quad \text{on}\quad \G_\e.
\end{align*}
The associated integral identity with $v_\e$ taken as the test function is
\begin{equation}\label{3.1}
\hn(v_\e,v_\e)-\iu(v_\e,v_\e)_{L_2(\Om_\e)}=- \left(\frac{\p u_0}{\p\nu^\e},v_\e\right)_{L_2(\G_\e)}.
\end{equation}
Since
\begin{equation*}
\nu^\e=\frac{1}{\sqrt{1+
 \e^{-2}\eta^2(\e)\left(b'(x_1\e^{-1})\right)^2}} \left(- \e^{-1}\eta(\e)b'(x_1\e^{-1}),1\right),
\end{equation*}
we have
\begin{align}\label{4.23}
&\left(\frac{\p u_0}{\p\nu^\e},v_\e\right)_{L_2(\G_\e)}=\int\limits_{\mathds{R}} \big(\e^{-1}\eta b'(x_1\e^{-1}) w_1(x)-w_2(x)\big)\overline{v}_\e(x)\Big|_{x_2=\eta b(x_1\e^{-1})}\di x_1,
\\
&w_j^\e:=\sum\limits_{i=1}^{2} A_{ij} \frac{\p u_0}{\p x_i}+\overline{A}_j u_0.\nonumber
\end{align}
Denote
\begin{equation*}
w_3(x):=\int\limits_{b_*\eta}^{x_2} w_1(x_1,t)\overline{v}_\e(x_1,t)\di t,
\end{equation*}
where, we recall, $b_*:=\max\limits_{[0,1]} b$. The identity
\begin{equation*}
\e^{-1}\eta w_1\left(x_1,\eta b(x_1\e^{-1})\right)b'(x_1\e^{-1})=\frac{d}{dx_1} w_3\left(x_1,\eta b(x_1\e^{-1})\right) - \frac{\p w_3}{\p x_1}\left(x_1,\eta b(x_1\e^{-1})\right),
\end{equation*}
implies
\begin{align*}
\bigg|&\e^{-1}\eta \int\limits_{\mathds{R}}b'(x_1\e^{-1})  w_1 \overline{v}_\e\Big|_{x_2=\eta b(x_1\e^{-1})}\di x_1\bigg|
\\
&=\bigg|\int\limits_{\mathds{R}}
\frac{\p w_3}{\p x_1}\Big|_{x_2=\eta b(x_1\e^{-1})}\di x_1
\bigg|=\bigg|\int\limits_{\mathds{R}}\di x_1\int\limits_{\eta b_*}^{\eta b(x_1\e^{-1})} \frac{\p w^\e_1}{\p x_1}(x)\overline{v}_\e\di x_2\bigg|
\\
&\leqslant C\left(\|u_0\|_{\H^2(\Om_\e)}\|v_\e\|_{L_2(\Om_\e\cap\{x: x_2<b_*\eta\})}+\|u_0\|_{\H^1(\Om_\e)\cap\{x: x_2<b_*\eta \}}\|v_\e\|_{\H^1(\Om_\e)}
\right).
\end{align*}
Applying Lemma~\ref{lm4.1} with $u=u_0$, $u=\frac{\p u_0}{\p x_i}$, $u=v_\e$, we obtain
\begin{align*}
&\|v_\e\|_{L_2(\Om_\e\cap\{x: x_2<\eta b_*\})}\leqslant C\eta^{1/2}\|v_\e\|_{\H^1(\Om_\e)},
\\
&\|u_0\|_{\H^1(\Om_\e)\cap\{x: x_2<\eta b_*\}}\leqslant C\eta^{1/2} \|u_0\|_{\H^2(\Om_0)}^2.
\end{align*}
Proceeding as \cite[Ch. I\!I\!I, Sec. 7,8]{LU}, \cite[Lm. 2.2]{B-AA-08}, one can estimate $u_0$,
\begin{equation}\label{4.7}
\|u_0\|_{\H^2(\Om_0)}\leqslant C\|f\|_{L_2(\Om_0)}.
\end{equation}
Thus, by last three inequalities, % and (\ref{4.7}),
\begin{equation}\label{4.24}
\bigg|\e^{-1}\eta \int\limits_{\mathds{R}}b'(x_1\e^{-1}) w_1 \overline{v}_\e \Big|_{x_2=\eta b(x_1\e^{-1})}\di x_1\bigg|\leqslant C\eta^{1/2} \|f\|_{L_2(\Om_0)}\|v_\e\|_{\H^1(\Om_\e)}.
\end{equation}

We estimate the second term in the right hand side of (\ref{4.23}) as follows,
\begin{equation}\label{3.8}
\bigg|\int\limits_{\mathds{R}}  w_2(x) \overline{v}_\e(x)\Big|_{x_2=\eta b(x_1\e^{-1})}\di x_1
\bigg|\leqslant \big\|w_2(\cdot,\eta b(\cdot\,\e^{-1}))
 \big\|_{L_2(\mathds{R})} \big\|v_\e(\cdot,\eta b(\cdot\,\e^{-1}))
 \big\|_{L_2(\mathds{R})}.
\end{equation}
In view of the boundary condition for $u_0$ on $\G_0$, the function $w_2$ vanishes at $x_2=0$. Since it also belongs to $\H^1(\Om_0)$, by analogy with Lemma~\ref{lm3.5} one can prove easily that
\begin{align*}
& \big\|w_2(\cdot,\eta b(\cdot\,\e^{-1}))
 \big\|_{L_2(\mathds{R})}^2\leqslant C\eta \|u_0(x_1,\cdot)\|_{\H^2(0,d)}^2\quad\text{for a.e.}\  x_1\in\mathds{R},
\\
& \big\|w_2(\cdot,\eta b(\cdot\,\e^{-1}))
 \big\|_{L_2(\mathds{R})}\leqslant C\eta^{1/2} \|u_0\|_{\H^2(\Om_0)}.
\end{align*}
The latter estimate, Lemma~\ref{lm4.1}, (\ref{4.7}), (\ref{4.23}), (\ref{4.24}), and (\ref{3.8}) yield
\begin{equation*}
\bigg|\bigg(\frac{\p u_0}{\p\nu^\e},u_\e\bigg)_{L_2(\G_\e)}\bigg|\leqslant C\eta^{1/2}\|f\|_{L_2(\Om_\e)}\|v_\e\|_{\H^1(\Om_\e)}.
\end{equation*}
By Lemma~\ref{lm4.2} it implies
\begin{equation*}
\|v_\e\|_{\H^1(\Om_\e)}\leqslant C\eta^{1/2} \|f\|_{L_2(\Om_\e)}.
\end{equation*}
The proof is complete.
\end{proof}

\begin{proof}[Proof of Theorem~\ref{th2.2}] We follow the same lines as in the previous proof. Denote $u_\e:=(\Hr-\iu)^{-1}f$, $u_0:=(\Hor1-\iu)^{-1}f$, $v_\e:=u_\e-u_0$. By analogy with (\ref{3.1}) we get
\begin{equation}\label{3.23}
\hn(v_\e,v_\e)-\iu(v_\e,v_\e)_{L_2(\Om_\e)}=- \left(\left(\frac{\p\hphantom{\nu} }{\p\nu^\e}+a\right)u_0,v_\e\right)_{L_2(\G_\e)}.
\end{equation}
It follows from (\ref{4.23}) that
\begin{equation}\label{3.20}
\begin{aligned}
&\left(\left(\frac{\p\hphantom{\nu} }{\p\nu^\e}+a\right)u_0,v_\e\right)_{L_2(\G_\e)}=
\e^{-1}\eta\int\limits_{\mathds{R}} b'(x_1\e^{-1}) w_1(x) \overline{v}_\e(x)\Big|_{x_2=\eta b(x_1\e^{-1})}\di x_1
\\
&\hphantom{2}+  \int\limits_{\mathds{R}} \left(a\sqrt{1+\e^{-2}\eta^2 \big(b'(x_1\e^{-1})\big)^2
}u_0(x)-w_2(x)\right) \overline{v}_\e(x) \Big|_{x_2=\eta b(x_1\e^{-1})}\di x_1,
\end{aligned}
\end{equation}
and the first term in the right hand side can be again estimated by (\ref{4.24}).

Due to the boundary condition on $\G_0$ in   operator $\Hor1$ we have $w_2-a_0 u_0=0$ on $\G_0$ and by analogy with Lemma~\ref{lm3.5} one can make sure that
\begin{equation}\label{3.21}
\big\|w_2\big|_{x_2=\eta b(\cdot\e^{-1})}-w_2\big|_{x_2=0}
\big\|_{L_2(\mathds{R})}\leqslant C\eta^{1/2}\|u_0\|_{\H^2(\Om_0)}.
\end{equation}
Hence, to estimate the second term in the right hand side of (\ref{3.23}), it is sufficient to estimate
\begin{equation}\label{3.22}
\int\limits_{\mathds{R}}  \left(
a\big(x_1,\eta b(x_1\e^{-1})\big) \sqrt{1+\e^{-2}\eta^2 \big(b'(x_1\e^{-1})\big)^2
} - a_0(x_1)\right)u_0(x_1,0) \overline{v}_\e\big(x_1,\eta b(x_1\e^{-1})\big)\di x_1.
\end{equation}
Considering separately the cases $\a=0$ and $\a\not=0$,  it is easy to see that
\begin{equation*}
\left|
a\big(x_1,\eta b(x_1\e^{-1})\big) \sqrt{1+\e^{-2}\eta^2 \big(b'(x_1\e^{-1})\big)^2
} - a_0(x_1)\right|\leqslant C\big(\eta^{1/2}+|\e^{-2}\eta^2-\a^2|\big).
\end{equation*}
Thus, by Lemma~\ref{lm4.1} and (\ref{4.7}), integral (\ref{3.22}) can be estimated  from above by
$C\big(\eta^{1/2}+|\e^{-2}\eta^2-\a^2|\big)
\|f\|_{L_2(\Om_0)} \|v_\e\|_{\H^1(\Om_\e)}$. Together with (\ref{3.20}), (\ref{4.24}), (\ref{3.21}) it leads us to the final estimate
\begin{equation*}
\left|
 \left(\left(\frac{\p\hphantom{\nu} }{\p\nu^\e}+a\right)u_0,v_\e\right)_{L_2(\G_\e)}
\right|\leqslant C \big(\eta^{1/2}+|\e^{-2}\eta^2-\a^2|\big) \|f\|_{L_2(\Om_0)} \|v_\e\|_{\H^1(\Om_\e)}.
\end{equation*}
Substituting this estimate into (\ref{3.23}) and applying Lemma~\ref{lm4.2}, we complete the proof.
\end{proof}

Let us show that the estimates in the proven theorems are sharp. Assume the differential expression for $\Hr$ is the negative Laplacian, $2\leqslant b(t)\leqslant 3$, the function $a$ is constant and $a_0$ is determined by (\ref{2.10a}) via $a$. Let
\begin{equation*}
F=F(x_2)=\left\{
\begin{aligned}
&0,\quad \eta<x_2<d,
\\
&1,\quad 0<x_2<\eta,
\end{aligned}
\right.
\end{equation*}
and $U=U(x_2)$ be the solution to the boundary value problem
\begin{equation*}%\l%abel{4.6}
-U''-\iu U=-\iu F\quad \text{in}\quad (0,d),\qquad U'(0)-a_0 U(0)=0,\qquad U(d)=0.
\end{equation*}
The function $U$ can be found explicitly,
\begin{equation*}
U(x_2)=-\frac{k\sin k\eta+a_0(1-\cos k\eta)}{k\cos kd+a_0\sin kd}\sin k(d-x_2),
\end{equation*}
as $\eta<x_2<d$, and
\begin{equation*}
U(x_2)=1+\frac{a_0}{k}\sin kx_2-\frac{k\cos k(d-\eta)+a_0\sin k d)}{k\cos kd+a_0\sin kd}\left(\cos kx_2+\frac{a_0}{k}\sin kx_2
\right),
\end{equation*}
as $0<x_2<\eta$, where
$k:=\E^{\iu\pi/4}$. By $\vp=\vp(x_1)$  we denote an arbitrary function in $C_0^\infty(R)$ normalized in $L_2(\mathds{R})$, and let $u_0(x):=
\vp(\eta x_1)U(x_2)$. The latter function satisfies $u_0=(\Hor1-\iu)^{-1}f$, where
\begin{equation*}
f=(-\D-\iu)u_0=f_1+f_2,\quad f_1:=-\iu
\vp(\eta x_1)F(x_2),\quad f_2:=-\eta^2\vp''(\eta x_1)U(x_2).
\end{equation*}
It is straightforward to check that
\begin{equation}\label{4.5}
\begin{aligned}
&\|f_1\|_{L_2(\Om_0)}=1,\quad \|f_2\|_{L_2(\{x: \eta<x_2<d\})}\leqslant C\eta^{5/2},
\\
&\|u_0\|_{L_2(\{x: 3\eta<x_2<d\})}\geqslant C\eta^{1/2}\|f\|_{L_2(\Om_0)},
\\
&\|\nabla u_0\|_{L_2(\{x: 3\eta<x_2<d\})}\geqslant C\eta^{1/2}\|f\|_{L_2(\Om_0)}.
\end{aligned}
\end{equation}
By Lemma~\ref{lm4.2} we have the a priori estimate
\begin{equation*}%\l%abel{4.7a}
\|(\Hr-\iu)^{-1}f\|_{\H^1(\Om_\e)}\leqslant C\|f\|_{L_2(\Om_\e)}
\end{equation*}
uniform in $\e$. Since $f=f_2$ on $\Om_\e$, by this estimate
and (\ref{4.5}) we get
\begin{align*}
&\|(\Hr-\iu)^{-1}f\|_{\H^1(\Om_\e)}\leqslant C\eta^{5/2}\|f\|_{L_2(\Om_0)},
\\
&\|(\Hr-\iu)^{-1}f-(\Hor1-\iu)^{-1}f\|_{L_2(\Om_\e)}\geqslant C\eta^{1/2}\|f\|_{L_2(\Om_0)},
\\
&\big\|\nabla\big((\Hr-\iu)^{-1}f- (\Hor1-\iu)^{-1}f\big)\big\|_{L_2(\Om_\e)}\geqslant C\eta^{1/2}\|f\|_{L_2(\Om_0)}.
\end{align*}
Thus, as $a=0$, the adduced example proves the sharpness of the estimate in Theorem~\ref{th2.3}. For arbitrary $a$ it proves the sharpness of the term $\eta^{1/2}$ in the estimate in Theorem~\ref{th2.2}. The other term, $|\e^2\eta^{-2}-\a^2|$ is also sharp. Indeed, if we  define the operator $\widetilde{\mathcal{H}}^{\mathrm{R}}_0$ in the same way as $\Hor1$ but replacing $\a^2$ by $\e^2\eta^{-2}$ in (\ref{2.10a}), reproducing the proof of Theorem~\ref{th2.2} we can make sure that
\begin{equation*}
\|(\Hr-\iu)^{-1}f-(\widetilde{\mathcal{H}}^{\mathrm{R}}_0-\iu)^{-1} f\|_{\H^1(\Om_\e)}\leqslant
C\eta^{1/2} \|f\|_{L_2(\Om_0)}.
\end{equation*}
But then operator $\Hor1$ can be regarded as a regular perturbation of $\Hor1$ and hence
\begin{equation*}
\|(\Hor1-\iu)^{-1}f-(\widetilde{\mathcal{H}}^{\mathrm{R}}_0-\iu)^{-1} f\|_{\H^1(\Om_\e)}\sim
|\e^2\eta^{-2}-\a^2|\|f\|_{L_2(\Om_0)}.
\end{equation*}
Therefore, the estimate in Theorem~\ref{th2.2} is sharp.

\section{ Robin condition on relatively high oscillating boundary}

In this section we prove Theorems~\ref{th2.4},~\ref{th2.5}.
Throughout the proofs we indicate by $C$ various inessential constants independent of $\e$, $x$, and $f$.

\begin{proof}[Proof of Theorem~\ref{th2.4}]
Given a function  $f\in L_2(\Om_0)$, we let
\begin{equation*}%\l%abel{5.0}
u_\e:=(\Hd-\iu)^{-1}f,\quad u_0:=(\Hod-\iu)^{-1}f,\quad v_\e:=u_\e-(1-K)u_0,
\end{equation*}
where the function $K$ is introduced by (\ref{3.33}). We remind that the function $1-K$ vanishes as $x_2<b_*\eta$.

We write the integral identity for $u_\e$ choosing $v_\e$ as the test function,
\begin{equation}\label{5.1}
\hr(u_\e,v_\e)-\iu(u_\e,v_\e)_{L_2(\Om_\e)}=(f,v_\e)_{L_2(\Om_\e)},
\end{equation}
and that for $u_0$ with the test function $(1-K)v_\e$ extended by zero in $\Om_0\setminus\Om_\e$,
\begin{equation}\label{5.2}
\hod\big(u_0,(1-K)v_\e\big)-\iu(u_0,(1-K)v_\e)_{L_2(\Om_\e)}= (f,(1-K)v_\e)_{L_2(\Om_\e)}.
\end{equation}
We observe that
\begin{equation*}
(au,(1-K)v)_{L_2(\G_\e)}=0
\end{equation*}
for all $u,v\in\H^1(\Om_\e)$. Bearing this fact in mind, we reproduce the arguments used in proving (\ref{3.6a}) and check easily that
\begin{align*}
\hod\big(u_0,(1-K)v_\e\big)=&\hr\big((1-K)u_0,v_\e\big) +  \left(A_2 u_0,v_\e \frac{\p K}{\p x_2}\right)_{L_2(\Om_\e)}-  \left(u_0,A_2 \frac{\p  K}{\p x_2}v_\e\right)_{L_2(\Om_\e)}
\\
&-\sum\limits_{ j=1}^{2} \left(A_{2j}\frac{\p u_0}{\p x_j}, v_\e\frac{\p  K}{\p x_2}\right)_{L_2(\Om_\e)}
+\sum\limits_{i =1}^{2}\left(A_{i2}u_0\frac{\p K}{\p x_2},\frac{\p v_\e}{\p x_i}\right)_{L_2(\Om_\e)}.
\end{align*}
We substitute this identity into (\ref{5.2}) and calculate the difference of the result and (\ref{5.1}),
\begin{equation}\label{5.3}
\begin{aligned}
\hr(v_\e,v_\e)-&\iu\|v_\e\|_{L_2(\Om_\e)}^2=(f,Kv_\e)_{L_2(\Om_\e)} + \left(A_2 u_0,v_\e \frac{\p K}{\p x_2}\right)_{L_2(\Om_\e)}
\\
&-  \left(u_0,A_2 \frac{\p  K}{\p x_2}v_\e\right)_{L_2(\Om_\e)}
-\sum\limits_{ j=1}^{2} \left(A_{2j}\frac{\p u_0}{\p x_j}, v_\e\frac{\p  K}{\p x_2}\right)_{L_2(\Om_\e)}
\\
&+\sum\limits_{i =1}^{2}\left(A_{i2}u_0\frac{\p K}{\p x_2},\frac{\p v_\e}{\p x_i}\right)_{L_2(\Om_\e)}.
\end{aligned}
\end{equation}
By the definition of $K$, (\ref{3.35}), and Lemma~\ref{lm3.5} we have
\begin{equation}
\begin{aligned}
&\bigg|(f,Kv_\e)_{L_2(\Om_\e)} + \left(A_2 u_0,v_\e \frac{\p K}{\p x_2}\right)_{L_2(\Om_\e)}
\\
&-   \left(u_0,A_2 \frac{\p  K}{\p x_2}v_\e\right)_{L_2(\Om_\e)}
+\sum\limits_{i =1}^{2}\left(A_{i2}u_0\frac{\p K}{\p x_2},\frac{\p v_\e}{\p x_i}\right)_{L_2(\Om_\e)}\bigg|
\\
&\leqslant C\bigg(\|f\|_{L_2(\Om^\eta)} \|v_\e\|_{L_2(\Om^\eta)}+
\eta^{-1}\|u_0\|_{L_2(\widetilde{\Om}^\eta)} \|v_\e\|_{\H^1(\widetilde{\Om}^\eta)}\bigg)
\\
&\leqslant C\eta^{1/2}  \|f\|_{L_2(\Om_0)}\|v_\e\|_{\H^1(\Om_\e)},
\end{aligned}\label{5.10}
\end{equation}
where, $\widetilde{\Om}^\eta:=\{x: b_*\eta<x_2<(b_*+1)\eta\}$, and, we recall, $\Om^\eta=\Om_\e\cap\{x: 0<x_2<(b_*+1)\eta\}$.

Denote
\begin{equation*}
\widetilde{b}(t):=\int\limits_{0}^{t} |b'(z)|\di z-t\int\limits_{0}^{1} |b'(z)|\di z.
\end{equation*}
This function is continuous and $1$-periodic. It  satisfies the identity
\begin{equation*}
\widetilde{b}'(t):=|b'(t)|-\int\limits_{0}^{1} |b'(z)|\di z.
\end{equation*}

Hence,
\begin{equation}
\begin{aligned}
&\int\limits_{0}^{1}|b'(t)|\di t \sum\limits_{ j=1}^{2} \left(A_{2j}\frac{\p u_0}{\p x_j}, v_\e\frac{\p  K}{\p x_2}\right)_{L_2(\Om_\e)}
=\sum\limits_{ j=1}^{2} \left(A_{2j}\frac{\p u_0}{\p x_j},|b'|\frac{\p K}{\p x_2}\overline{v}_\e \right)_{L_2(\widetilde{\Om}^\eta)}
\\
&-\e\sum\limits_{j=1}^{2} \int\limits_{\Om_\e} |\widetilde{b}| \frac{\p K}{\p x_2}\frac{\p}{\p x_1} A_{2j}\frac{\p u_0}{\p x_j}\overline{v}_\e\di x=-\sum\limits_{j=1}^{2}  \left(A_{2j}\frac{\p u_0}{\p x_j},|b'| \overline{v}_\e \right)_{L_2(\G^\eta)}
\\
&- \sum\limits_{j=1}^{2} \int\limits_{\widetilde{\Om}^\eta} K |b'|
\frac{\p}{\p x_2} \left(A_{2j}\frac{\p u_0}{\p x_j}\overline{v}_\e\right)
\di x-\e\sum\limits_{j=1}^{2} \int\limits_{\widetilde{\Om}^\eta}
|\widetilde{b}|\frac{\p K}{\p x_2}\frac{\p}{\p x_1} \left( A_{2j}\frac{\p u_0}{\p x_j}\overline{v}_\e\right)\di x,
\end{aligned}\label{4.8}
\end{equation}
where $b'=b'(x_1\e^{-1})$, $\widetilde{b}=\widetilde{b}(x_1\e^{-1})$, $\G^\eta:=\{x: x_2=b_*\eta\}$. Since $b(t)$ is not identically constant, in view of Lemma~\ref{lm3.5},  and the definition of $K$ we get
\begin{equation}\label{5.13}
\begin{aligned}
\bigg|\sum\limits_{ j=1}^{2} \left(A_{2j}\frac{\p u_0}{\p x_j}, v_\e\frac{\p  K}{\p x_2}\right)_{L_2(\Om_\e)}\bigg| \leqslant & C\|f\|_{L_2(\Om_0)}\||b'|^{1/2} v_\e\|_{L_2(\G^\eta)}
\\
&+ C(\e\eta^{-1/2}+\eta^{1/2})\|f\|_{L_2(\Om_0)} \|v_\e\|_{\H^1(\Om_\e)}.
\end{aligned}
\end{equation}
The identity
\begin{align}
&\||b'|^{1/2} v_\e\|_{L_2(\G^\eta)}^2=\int\limits_{\mathds{R}} \di x_1 \int\limits_{b_*\eta}^{\eta b(x_1\e^{-1})} |b'(x_1\e^{-1})| \frac{\p|v_\e|^2}{\p x_1}\di x_2+q_\e, \label{4.3}
\\
&q_\e:=\int\limits_{\mathds{R}} |b'(x_1\e^{-1})| \big|v_\e(x_1,\eta b(x_1\e^{-1}))\big|^2\di x_1,\nonumber
\end{align}
and Lemma~\ref{lm3.5} imply
\begin{equation}\label{4.1}
\begin{aligned}
\||b'|^{1/2} v_\e\|_{L_2(\G^\eta)} \leqslant C\eta^{1/2}\|v_\e\|_{\H^1(\Om_\e)} +q_\e^{1/2}.
\end{aligned}
\end{equation}
In view of the assumption (\ref{2.11}) the boundary term in the definition of the form $\hr$ can be estimated as
\begin{equation}\label{5.4}
(av_\e,v_\e)_{L_2(\G_\e)}\geqslant  c_1\int\limits_{\mathds{R}} \big|v_\e(x_1,\eta b(x_1\e^{-1}))\big|^2\sqrt{1+\e^{-2}\eta^2 \big(b'(x_1\e^{-1})\big)^2}\di x_1
\geqslant \frac{c_1\eta}{\e} q_\e.
\end{equation}
This inequality, (\ref{3.4a}), (\ref{5.3}), (\ref{5.10}), (\ref{5.13}), (\ref{4.1}), (\ref{2.9}) yield
\begin{align*}
\frac{\eta}{\e} q_\e + \|v_\e\|_{\H^1(\Om_\e)}^2
&\leqslant
C\left(\eta^{1/2}\|v_\e\|_{\H^1(\Om_\e)}+
 q_\e^{1/2}
\right)\|f\|_{L_2(\Om_0)}
\\
&\leqslant C(\eta^{1/2}+\e^{1/2}\eta^{-1/2})\left( \|v_\e\|_{\H^1(\Om_\e)}^2+\frac{\eta}{\e}
  q_\e
\right)^{1/2}\|f\|_{L_2(\Om_0)}.
\end{align*}
It follows that
\begin{equation*}
\|v_\e\|_{\H^1(\Om_\e)}^2\leqslant C(\eta^{1/2}+\e^{1/2}\eta^{-1/2})\|f\|_{L_2(\Om_0)}.
\end{equation*}
It remains to employ (\ref{3.13}) to complete the proof.
\end{proof}

\begin{proof}[Proof of Theorem~\ref{th2.5}]
The proof of this theorem is similar to the previous one up to a substantial modification. All the arguments of the previous proof remain true up to inequality (\ref{4.1}), while estimate (\ref{5.4}) is no longer valid since we replace assumption (\ref{2.11}) by  (\ref{2.12}).   And the aforementioned modification  is  a new estimate substituting (\ref{5.4}).

Given any $\d>0$, we split set $\G^\eta$ into two parts, $\G^\eta=\G^\eta_\d\cup\G^{\eta,\d}$,
\begin{equation*}
\G^\eta_\d:=\{x: a(x_1,b_*\eta)> \d, x_2=b_*\eta\},\quad \G^{\eta,\d}:=\{x: a(x_1,b_*\eta)\leqslant \d, x_2=b_*\eta\},
\end{equation*}
and let $\g_\d:=\{x_1\in \mathds{R}: (x_1,b_*\eta)\in \G^\eta_\d\}$. We observe that
\begin{equation}\label{4.2}
(a v_\e,v_\e)_{L_2(\G_\e)}\geqslant \frac{\eta}{\e} \int\limits_{\g} a(x_1,\eta b(x_1\e^{-1})) |b'(x_1\e^{-1})| \big|v_\e(x_1,\eta b(x_1\e^{-1}))\big|^2\di x_1:=\frac{\eta}{\e} \mathfrak{q}_{\e,\d}.
\end{equation}
By analogy with (\ref{4.3}), (\ref{4.1}) we obtain
\begin{equation}\label{4.4}
\||b'|^{1/2}v_\e\|_{L_2(\G^\eta_\d)}\leqslant C\eta^{1/2}| \|v_\e\|_{\H^1(\Om_\e)}+\d^{-1/2}\mathfrak{q}_{\e,\d}^{1/2}.
\end{equation}

The next auxiliary lemma will allow us to estimate $\|v_\e\|_{L_2(\G^{\eta,\d})}$.

\begin{lemma}\label{lm5.2}
Let $v\in\H^1(\square)$, where $\square:=\{x: d_-<x_1<d_+, b_*\eta<x_2<d\}$, $d_+>d_-$, $d_\pm$ are some constants. Denote $\Xi_\mu:=\{x: |x_1-d_0|<\mu, x_2=b_*\eta\}$ and suppose that there exists a positive constant $c$ independent of $\mu$ such that $d_0-d_-\geqslant c$, $d_+-d_0\geqslant c$. Then for sufficiently small $\mu$ there exists a positive constant $C$ independent of $\mu$ and $v$ but dependent on $c$ such that the inequality
\begin{equation}\label{4.15}
\|v\|_{L_2(\Xi_\mu)}\leqslant C\mu^{1/2}|\ln\mu|^{1/2} \|v\|_{\H^1(\square)}
\end{equation}
holds true.
\end{lemma}

\begin{proof}
We expand $v$ in a Fourier series
\begin{equation}\label{5.16}
v(x)=\sum\limits_{m,n=0}^{\infty} c_{mn} \cos\frac{\pi n}{d_+-d_-}(x_1-d_-)\cos\frac{\pi m}{(d-\eta)} (x_2-\eta)
\end{equation}
converging in $\H^1(\square)$. In view of the Parseval identity one has
\begin{equation}\label{5.17}
\sum\limits_{m,n=0}^{\infty} |c_{mn}|^2(m^2+n^2)\leqslant C\|v\|_{\H^1(\square)}^2.
\end{equation}
Due to the embedding of $\H^1(\square)$ into $L_2(\Xi_\mu)$, we can employ (\ref{5.16}) to calculate $\|v\|_{L_2(\Xi_\mu)}$,

\begin{align*}
\|v\|_{L_2(\Xi_\mu)}^2=&
\sum\limits_{m,n,p,q=0}^{\infty} c_{mn}\overline{c_{pq}} \int\limits_{d_0-\mu}^{d_0+\mu}
\cos\frac{\pi n}{d_+-d_-}(x_1-d_-)
\cos\frac{\pi p}{d_+-d_-}(x_1-d_-)\di x_1
\\
=&\frac{d_+-d_-}{\pi}\sum\limits_{m,n,p,q=0}^{\infty} c_{mn}\overline{c_{pq}} \left(\sin\frac{\pi (n+p)\mu}{d_+-d_-}
\frac{\cos\pi(n+p)}{n+p}\right.
\\
&\hphantom{\frac{d_+-d_-}{\pi}\sum\limits_{m,n,p,q=0}^{\infty} c_{mn}\overline{c_{pq}} \Bigg(}
\left.-\sin\frac{\pi (n-p)\mu}{d_+-d_-}
\frac{\cos\pi(n-p)}{n-p}\right),
\end{align*}
where $\sin\frac{\pi (n-p)\mu}{d_+-d_-}/(n-p)$ is to be replaced by $\pi\mu/d$ as $n=p$. We employ Cauchy-Schwarz inequality and  the inequality
\begin{equation*}
\sin^2 t\leqslant \frac{t^2}{1+t^2},\quad t\geqslant 0,
\end{equation*}
and by (\ref{5.17}) we obtain
\begin{align*}
&\|v\|_{L_2(\Xi_\mu)}^4\leqslant \sum\limits_{m,n,p,q=0}^{\infty} |c_{mn}|^2 |c_{pq}|^2 (m^2+n^2)(p^2+q^2)
\\
&\sum\limits_{m,n,p,q=0}^{\infty}\frac{2}{(m^2+n^2)(p^2+q^2)}
\left(
\sin^2\frac{\pi (n+p)\mu}{d_+-d_-}
\frac{1}{(n+p)^2}+\sin^2\frac{\pi (n-p)\mu}{d_+-d_-}
\frac{1}{(n-p)^2}
\right)
\\
&\leqslant C\mu^2 \|v\|_{\H^1(\square)}^4 \sum\limits_{m,n,p,q=0}^{\infty} \frac{1}{(m^2+n^2)(p^2+q^2)}\left(\frac{1}{1+\mu^2(n+p)^2}+\frac{1} {1+\mu^2(n-p)^2}\right)
\\
&\leqslant C\mu^2 \|v\|_{\H^1(\square)}^4
\sum\limits_{n,p=0}^{\infty}\left(\frac{1}{1+\mu^2(n+p)^2} +\frac{1} {1+\mu^2(n-p)^2}\right) \int\limits_{1}^{+\infty}\frac{dz}{z^2+n^2}\int\limits_{1}^{+\infty} \frac{dz}{z^2+p^2}
\\
&\leqslant C\mu^2 \|v\|_{\H^1(\square)}^4
\sum\limits_{n,p=0}^{\infty}\frac{1}{1+nq} \left(\frac{1}{1+\mu^2(n+p)^2}+\frac{1} {1+\mu^2(n-p)^2}\right).
\end{align*}
In the last sum we extract the terms for $(n,p)=(0,0)$, $(n,p)=(0,1)$, and $(n,p)=(1,0)$. Then we replace the remaining summation by the integration and estimate in this way the sum by a two-dimensional integral,
\begin{align*}
\|v\|_{L_2(\Xi_\mu)}^4\leqslant & 3\mu^2\|v\|_{\H^1(\square)}^4
\\
&+ C\mu^2\|v\|_{\H^1(\square)}^4\int\limits_{\genfrac{}{}{0 pt}{}{z_1^2+z_2^2>3}{z_1, z_2>0}} \left(\frac{1}{1+\mu^2(z_1+z_2)^2}+\frac{1} {1+\mu^2(z_1-z_2)^2}\right)\frac{\di z_1 \di z_2}{1+z_1z_2}.
\end{align*}
Passing to the polar coordinates $(r,\tht)$ associated with $(z_1,z_2)$, we get
\begin{align*}
\int\limits_{\genfrac{}{}{0 pt}{}{z_1^2+z_2^2>3}{z_1, z_2>0}} & \left(\frac{1}{1+\mu^2(z_1+z_2)^2}+\frac{1} {1+\mu^2(z_1-z_2)^2}\right)\frac{\di z_1 \di z_2}{1+z_1z_2}
\\
&\leqslant  2\int\limits_{\sqrt{3}}^{+\infty}\int\limits_{0}^{\pi/2} \left(\frac{1}{1+\mu^2r^2(1+\sin 2\tht)}+\frac{1} {1+\mu^2r^2(1-\sin 2\tht)}\right)\frac{r\di r\di \tht}{1+r^2\sin 2\tht}
\\
&\leqslant C\int\limits_{3}^{+\infty} \left(\frac{\ln \tau}{\tau (1+\mu^2\tau)} +\frac{\mu^2}{(1+\mu^2 \tau)^{3/2}}\right)\di \tau
\\
&=C\int\limits_{3\mu^2}^{+\infty} \left(\frac{\ln \tau-2\ln\mu}{\tau (1+\tau)} +\frac{1}{(1+ \tau)^{3/2}}\right)\di \tau\leqslant C\ln^2\mu.
\end{align*}
Two last formulas proves the desired estimate for $\|v\|_{L_2(\Xi_\mu)}$.
\end{proof}

We apply the proven lemma with $v=v_\e$, $d_-=X_n-c/2$, $d_+=X_n+c/2$, $d_0=X_n$ and sum the obtained inequalities over $n\in \mathds{Z}$. It gives the estimate for $\|v_\e\|_{L_2(\G^{\eta,\d})}$,
\begin{equation*}
\|v_\e\|_{L_2(\G^{\eta,\d})}\leqslant C\mu^{1/2}(\d)|\ln\mu(\d)|^{1/2}\|v_\e\|_{\H^1(\Om_\e)}.
\end{equation*}
This estimate and (\ref{4.4}) imply
\begin{align*}
 \||b'|^{1/2}v_\e\|_{L_2(\G^\eta)} &\leqslant \||b'|^{1/2}v_\e\|_{L_2(\G^\eta_\d)}+\||b'|^{1/2}v_\e\|_{L_2(\G^{\eta,\d})}
\\
&\leqslant C(\eta^{1/2}+\mu(\d)^{1/2}|\ln\mu(\d)|^{1/2}) \|v_\e\|_{\H^1(\Om_\e)} + \d^{-1/2}\mathfrak{q}_{\e,\d}^{1/2}.
\end{align*}
We substitute the obtained inequality and (\ref{5.10}), (\ref{5.13}), (\ref{4.2}) into (\ref{5.3}) and employ (\ref{3.4a}), (\ref{2.9}). It results in
\begin{align*}
\frac{\eta}{\e}\mathfrak{q}_{\e,\d}+\|v_\e\|_{\H^1(\Om_\e)} \leqslant  & C(\eta^{1/2}+\mu^{1/2}(\d)|\ln\mu(\d)|^{1/2})\|v_\e\|_{\H^1(\Om_\e)} \|f\|_{L_2(\Om_0)}
\\
& + C\d^{-1/2}\mathfrak{q}_{\e,\d}^{1/2} \|f\|_{L_2(\Om_0)}
\end{align*}
that leads us to the desired estimate
\begin{equation*}
\|v_\e\|_{\H^1(\Om_\e)}\leqslant C\big(\eta^{1/2} +\e^{1/2}\eta^{1/2}\d^{-1/2}+\mu^{1/2}(\d)|\ln\mu(\d)|^{1/2} \big) \|f\|_{L_2(\Om_0)}.
\end{equation*}
Together with  (\ref{3.13}) it completes the proof.
\end{proof}

Let us discuss the sharpness of the estimates in the proven theorems. We begin with Theorem~\ref{th2.4} and first show that the term $\e^{1/2}\eta^{-1/2}$ is sharp. In order to do it, we assume that $\eta^{1/2}\ll \e^{1/2}\eta^{-1/2}$, i.e., $\eta\e^{-1/2}\to+0$ as $\e\to+0$. We also suppose that the differential expression for  operator $\Hr$ is just the negative Laplacian. Then by (\ref{3.13})
\begin{equation*}
\|u_\e-u_0-v_\e\|_{\H^1(\Om_\e)}\leqslant C\eta^{1/2},
\end{equation*}
and it is sufficient to deal only with $v_\e$. Proceeding as in (\ref{5.3}), (\ref{5.10}), (\ref{4.8}), one can check that $v_\e$ can be represented as the sum $v_\e=v_\e^{(1)}+v_\e^{(2)}$, where $v_\e^{(i)}\in \Ho^1(\Om_\e,\G)$ solve the integral identities,
\begin{align*}
&(\nabla v_\e^{(1)},\nabla \vp)_{L_2(\Om_\e)}-\iu(v_\e^{(1)},\vp)_{L_2(\Om_\e)}+(a v^{(1)},\vp)_{L_2(\G_\e)}
\\
&=(f,K\vp)_{L_2(\Om_\e)}
+\left(u_0\frac{\p K}{\p x_2},\frac{\p\vp}{\p x_2}\right)_{L_2(\Om_\e)}
\\
&+ \left(\int\limits_{0}^{1}|b'(t)|\di t\right)^{-1}\left(
 \int\limits_{\widetilde{\Om}^\eta} K |b'|
\frac{\p}{\p x_2} \left( \frac{\p u_0}{\p x_2}\overline{\vp}_\e\right)
\di x+\e \int\limits_{\widetilde{\Om}^\eta}
|\widetilde{b}|\frac{\p K}{\p x_2}\frac{\p}{\p x_1} \left( \frac{\p u_0}{\p x_2}\overline{\vp}_\e\right)\di x
\right),
\end{align*}
and
\begin{equation}\label{4.9}
\begin{aligned}
&(\nabla v_\e^{(2)},\nabla \vp)_{L_2(\Om_\e)}-\iu(v_\e^{(2)},\vp)_{L_2(\Om_\e)}+(a v^{(2)},\vp)_{L_2(\G_\e)}
\\
&=\left(\int\limits_{0}^{1}|b'(t)|\di t\right)^{-1}
 \int\limits_{\mathds{R}} |b'(x_1\e^{-1})| \frac{\p u_0}{\p x_2}(x_1,\eta b(x_1\e^{-1}))\overline{\vp}(x_1,\eta b(x_1\e^{-1}))\di x_1,
\end{aligned}
\end{equation}
where $\vp\in \Ho^1(\Om_\e,\G)$. As in (\ref{5.3}), (\ref{5.10}), (\ref{4.8}), (\ref{4.3}), one can show that
\begin{equation*}
\|v_\e^{(1)}\|_{L_2(\Om_\e)}\leqslant C\eta^{1/2}\|f\|_{L_2(\Om_0)}
\end{equation*}
and it sufficient to show that $v_\e^{(2)}$ is indeed of order $\Odr(\e^{1/2}\eta^{-1/2})$. By Theorem~\ref{th2.4} and the latter estimate,
\begin{equation}\label{4.13}
\|v_\e^{(2)}\|_{L_2(\Om_\e)}\leqslant C(\eta^{1/2}+\e^{1/2}\eta^{-1/2})\|f\|_{L_2(\Om_0)},
\end{equation}
and hence
\begin{equation}\label{4.14}
\int\limits_{\mathds{R}} |v_\e^{(2)}(x_1,\eta b(x_1\e^{-1}))|^2\di x_1\leqslant C(\eta^{1/2}+\e^{1/2}\eta^{-1/2})\|f\|_{L_2(\Om_0)}.
\end{equation}

In (\ref{4.9}) we choose $\vp(x)=\frac{\p u_0}{\p x_2}(x)\chi_1(x_2)$, where cut-off function $\chi_1$
was introduced in the proof of Lemma~\ref{lm3.5}. By (\ref{4.13}) we then see that first two terms in the left hand side of (\ref{4.9}) are estimated by $(\eta^{1/2}+\e^{1/2}\eta^{-1/2})\|f\|_{L_2(\Om_0)}^2$. It is clear that we can choose $f$ so that $\|f\|_{L_2(\Om_0)}=1$ and
\begin{equation}
C_1\leqslant \int\limits_{\mathds{R}} |b'(x_1\e^{-1})| \left|\frac{\p u_0}{\p x_2}(x_1,\eta b(x_1\e^{-1}))\right|\di x_1\leqslant C_2
\end{equation}
uniformly in $\e$.   Hence, by (\ref{4.9}) with  $\vp(x)=\frac{\p u_0}{\p x_2}(x)\chi_1(x_2)$, (\ref{4.13}), (\ref{4.14}),
\begin{equation*}
\left|\int\limits_{\mathds{R}} |b'(x_1\e^{-1})| v_\e^{(2)} (x_1,\eta b(x_1\e^{-1}))  \frac{\p \overline{u_0}}{\p x_2}(x_1,\eta b(x_1\e^{-1})) \di x_1\right|\geqslant C\e\eta^{-1}.
\end{equation*}
In (\ref{4.9}) we let now $\vp=v_\e^{(2)}$ and get that $\|v_\e^{(2)}\|_{\H^1(\Om_\e)}$ is indeed of order $\Odr(\e^{1/2}\eta^{-1/2})$ and this order is sharp. Hence, the term $\e^{1/2}\eta^{-1/2}$ in  estimate (\ref{1.1}) is sharp.

To prove the sharpness of the other term, $\eta^{1/2}$, one needs to adduce some example, but we failed trying to find it.
Nevertheless, we know that such term is sharp under the hypotheses of Theorems~\ref{th2.1},~\ref{th2.3},~\ref{th2.2}. In Theorem~\ref{th2.4} the situation is more complicated since we have the oscillation is relatively high and we have Robin condition on the oscillating boundary. This is why the presence of the term $\eta^{1/2}$ in (\ref{1.1}) is reasonable and it seems to be sharp. At least in the framework of the technique we employed, this estimate can not be improved since all the inequalities in the proof are sharp.  We also note that similar estimate for the rate of the strong resolvent convergence in $L_2(\Om_\e)$-norm (not the uniform one!) established in \cite{CFP} is worse than (\ref{1.1}).

Estimate (\ref{1.2}) is worse than (\ref{1.1}) since we replace assumption (\ref{2.11}) by (\ref{2.12}). In this situation it is natural to have function $\mu$ involved in (\ref{1.2}). Here we can not adduce an example proving the sharpness of this estimate. On the other hand, we still have the term $\eta^{1/2}$. The term $\e^{1/2}\eta^{-1/2}\d^{-1/2}$  is similar to $\e^{1/2}\eta^{-1/2}$ in (\ref{1.1}). The presence of the factor $\d^{-1/2}$ shows how conditions (\ref{2.12}), (\ref{2.13}) spoil the estimate in comparison with (\ref{1.1}). The last term, $\mu(\d)|\ln\mu(\d)|$, also reflects the influence of the zeroes of $a$. It comes directly from (\ref{4.15}) which is a sharp inequality. To prove the latter fact, it is sufficient to make sure that for
\begin{equation*}
v(x):=\left\{
\begin{aligned}
&\hphantom{\ln|x-}1, \hphantom{\frac{-d_0b_*\eta)|}{\ln 3\mu}} \quad |x-(d_0,b_*\eta)|<3\mu,
\\
&\frac{\ln|x-(d_0,b_*\eta)|}{\ln 3\mu},\quad  |x-(d_0,b_*\eta)|>3\mu,
\end{aligned}
\right.
\end{equation*}
one has
\begin{equation*}
\|v\|_{L_2(\Xi_\mu)}\geqslant C\mu^{1/2}|\ln\mu|^{1/2}\|v\|_{\H^1(\square)}.
\end{equation*}

\section*{Acknowledgments}

The authors grateful to G.A.~Chechkin and T.A.~Mel'nyk for useful remarks and to the referee whose remarks allowed us to improve the original version of the paper.

D.B. was supported by RFBR and Federal Targeted Program ``Scientific and pedagogical staff of innovative Russia for 2009-2013'' (agreement  14.B37.21.0358) and by the Project ``Metodi asintotici per lo studio di alcuni funzionali e alcuni tipi di equazioni alle derivate parziali'', GNAMPA, 2011.


\begin{thebibliography}{99}

\bibitem{OIS}  O.A. Olejnik, A. S. Shamaev and G. A. Yosifyan, {Mathematical problems in elasticity and homogenization.} Studies in Mathematics and its Applications, {26}, North-Holland, Amsterdam etc. (1992)

\bibitem{DD} E.N. Dancer and D. Daners, Domain Perturbation for Elliptic Equations
Subject to Robin Boundary Conditions. {J. Diff. Eq.} {138} (1997) 86-132.

\bibitem{FHL} A. Friedman,  B. Hu and Y. Liu, A boundary value problem for the poisson equation with multi-scale oscillating Boundary. {J. Diff. Eq.} {137} (1997) 54-93.

\bibitem{FH} A. Friedman and B. Hu, A Non-stationary Multi-scale Oscillating Free Boundary for the Laplace and Heat Equations. {J.Diff. Eq.} {137} (1997) (47) 119-165.

\bibitem{ACG} Y. Amirat, G.A. Chechkin and
R.R. Gadyl'shin, Asymptotics for eigenelements of Laplacian
in domain with oscillating boundary:
multiple eigenvalues. {Appl. Anal.} {86} (2007) 873-897.

\bibitem{CC}  G. Chechkin and D. Cioranescu, Vibration of a thin plate with a ``rough'' surface.
Nonlinear Partial Differential Equations and their Applications, College de France Seminar Volume XIV,
{Studies in Mathematics and its Applications}, Elsevier, Amsterdam {31} (2002) 147-169.



\bibitem{ABCP} Y. Amirat, O. Bodart, G.A. Chechkin and
A.L. Piatnitski, Boundary homogenization in domains with randomly
oscillating boundary, {Stochastic Processes and their Applications } {121} (2011) 1-23.


\bibitem{ACG2} Y. Amirat, G.A. Chechkin and R.R. Gadyl'shin, Asymptotics of simple eigenvalues and eigenfunctions for the Laplace operator in a domain with oscillating boundary. {Computational Mathematics and Mathematical Physics}
{46} (2006), 97-110.




\bibitem{AB1} J. Arrieta, S. Bruschi, Boundary oscillations and nonlinear boundary conditions. {C.
R. Math. Acad. Sci. Paris} {343} (2006), 99-104.

\bibitem{AB2} J. Arrieta, S. Bruschi, Rapidly varying boundaries in equations with nonlinear boundary conditions. The case of a Lipschitz deformation. {Math. Models Methods Appl. Sci.} {17} (2007), 1555-1585.

\bibitem{AB3} J. Arrieta, S. Brushi, Very rapidly varying boundaries in equations with nonlinear
boundary conditions. The case of a non uniformly Lipschitz deformation. {Discrete Contin.
Dyn. Syst. Ser. B} {14} (2010), 327-351.


\bibitem{AC} J. Arrieta, A. Carvalho, Spectral convergence and nonlinear dynamics for reaction-
diffusion equations under perturbations of the domain. {J. Diff. Eq.} {199} (2004), 143-178.

\bibitem{BPC} A.G. Belyaev, A.L. Pyatnitskii  and G.A. Chechkin, Asymptotic behavior of a solution to a boundary value problem in a perforated domain with oscillating boundary. {Siberian Math. J.} {39} (1998) 621-644.

\bibitem{CFP}G.A. Chechkin, A. Friedman and A.L. Piatnitski, The Boundary-value Problem in Domains with
Very Rapidly Oscillating Boundary. {J. Math. Anal. Appl.} {231} (1999) 213-234.

\bibitem{KN} V.A. Kozlov and S.A. Nazarov, Asymptotics of the spectrum of the Dirichlet problem for the biharmonic operator in a domain with a deeply indented boundary. {St. Petersburg Math. J.} {22} (2011) 941–983.

\bibitem{ChCh} G.A. Chechkin and T.P. Chechkina, On homogenization of problems in domains of the ``Infusorium'' type. {J. Math. Sci., New York} {120} (2004) 1470-1482.


\bibitem{ChCh1} G.A. Chechkin, T.P. Chechkina. Homogenization Theorem for Problems in Domains of the ``Infusorian'' Type with Uncoordinated Structure. {J. Math. Sci.} {123} (2004) 4363-4380.

\bibitem{MP} E. Maru\v{s}i\'c-Paloka, Average of the navier's law on the rapidly oscillating boundary, {J. Math. Anal. Appl.} {259} (2001) 685-701.

\bibitem{Me} T.A. Mel'nik, Averaging of elliptic equations describing processes in strongly inhomogeneous thin perforated domains with rapidly changing thickness. {Akad. Nauk Ukr. SSR} {10} (1991) 15-18.

\bibitem{Mi} A. Mikeli\'c, Rough boundaries and wall laws. Qualitative properties of solutions to partial differential equations, Lecture notes of Necas Center for mathematical modeling, edited by E. Feireisl, P. Kaplicky and J. Malek,  Volume 5, Matfyzpress, Publishing House of the Faculty of Mathematics and Physics Charles University in Prague, Prague, 2009, 103-134.



\bibitem{Na3} S.A. Nazarov, The two terms asymptotics of the solutions of spectral problems with singular perturbations. {Math. USSR-Sb.} {69} (1991) 307-340.

\bibitem{Na2} S.A. Nazarov, Asymptotics of solutions and modelling the problems of elasticity theory in domains with rapidly oscillating boundaries. {Izv. Math.} {72} (2008).


\bibitem{Na4} S.A. Nazarov,  Dirichlet problem in an angular domain with rapidly oscillating boundary: Modeling of the problem and asymptotics of the solution. {St. Petersburg
    Math. J.} {19} (2008), 297-326.



\bibitem{GR} M.K. Gobbert and C.A. Ringhofer,
An Asymptotic Analysis for a Model of Chemical Vapor Deposition on a Microstructured Surface. {SIAM Jour. Appl. Math.} {58} (1998)  737-752.

\bibitem{CBS} R. Bunoiu, G. Cardone, T. Suslina, Spectral approach to homogenization of an elliptic operator periodic in some directions. {Math. Meth. Appl. Sci.} {34} (2011) 1075-1096.

\bibitem{CPZh} G. Cardone, S.E. Pastukhova, V.V. Zhikov,
Some estimates for nonlinear homogenization. {Rend. Accad. Naz. Sci. XL Mem. Mat. Appl.} {29} (2005) 101-110

\bibitem{CPP} G.Cardone, S.E. Pastukhova, C. Perugia, Estimates in homogenization of degenerate elliptic equations by spectral method, Asympt. Anal. 81 (2013) 189-209; DOI 10.3233/ASY-2012-1121.

\bibitem{PT} S.E. Pastukhova and R.N. Tikhomirov,
Operator Estimates in Reiterated and Locally Periodic Homogenization
{Dokl. Math.} {76} (2007) 548-553.

\bibitem{Pas} S.E. Pastukhova, Some Estimates from Homogenized Elasticity
Problems. {Dokl. Math.} {73} (2006) 102-106.


\bibitem{Bir1}
M.S. Birman, On homogenization procedure for periodic operators
near the edge of an internal gap. {St. Petersburg Math. J.}
{15} (2004) 507-513.

\bibitem{BS4}
M.S. Birman and T.A. Suslina, Homogenization of a multidimensional
periodic elliptic operator in a neighbourhood of the edge of the
internal gap. {J. Math. Sciences} {136} (2006) 3682-3690.


\bibitem{BS2}
M.S. Birman and T.A. Suslina, Homogenization with corrector term
for periodic elliptic differential operators {St. Petersburg
Math. J.} {17} (2006) 897-973.


\bibitem{BS5}
M.S. Birman and T.A. Suslina, Homogenization with corrector for
periodic differential operators. Approximation of solutions in the
Sobolev class $H^1({\mathbb{R}}^d)$. {St. Petersburg Math. J.}
{18} (2007) 857-955.


\bibitem{VS} E.S. Vasilevskaya and T.A. Suslina, Threshold approximations for a factorized selfadjoint operator family with the first and second correctors taken into account. {St. Petersbg. Math. J.} {23} (2012) 275-308.


\bibitem{Zh3} V.V. Zhikov, On operator estimates in homogenization theory.
{Dokl. Math.} {72} (2005) 534-538.

\bibitem{Zh6}  V.V. Zhikov, Spectral method in homogenization theory.
{Proc. Steklov Inst. Math.} {250} (2005) 85-94


\bibitem{Zh4} V.V. Zhikov, Some estimates from homogenization theory. {Dokl. Math.} {73}
(2006) 96-99.


\bibitem{Zh5} V.V. Zhikov, S.E. Pastukhova and S.V. Tikhomirova, On the
homogenization of degenerate elliptic equations. {Dokl. Math.}
{74} (2006) 716-720.




\bibitem{BBC2} D. Borisov, R. Bunoiu and G. Cardone, On a waveguide with frequently alternating boundary conditions: homogenized Neumann condition. {Ann. H. Poincar\'{e}}  {11} (2010) 1591-1627.

\bibitem{BBC} D. Borisov, R. Bunoiu and G. Cardone, On a waveguide with an infinite number of small windows. C.R. Mathematique {349} (2011) 53-56.

\bibitem{BBC3}  D. Borisov, R. Bunoiu and G. Cardone, Homogenization and asymptotics for a waveguide with an infinite number of closely located small windows. {J. of Math. Sci.} {176} (2011) 774-785.

\bibitem{BBC4} D. Borisov, R. Bunoiu and G. Cardone, Waveguide with non-periodically alternating Dirichlet and Robin conditions: homogenization and asymptotics. {Z. Ang. Math. Phys.}, to appear; doi: 10.1007/s00033-012-0264-2.

\bibitem{BC} D. Borisov and G. Cardone, Homogenization of the planar waveguide with frequently alternating boundary conditions. {J. of Phys. A: Mathematics and General} {42} (2009) 365205 (21pp).




\bibitem{LU} O.A. Ladyzhenskaya and N.N. Uraltseva, {Linear and quasilinear elliptic equations}. Academic Press, New York (1968).

\bibitem{B-AA-08} D. Borisov,  Asymptotics for the solutions of elliptic systems with fast oscillating coefficients. {St. Petersburg Math. J.} {20} (2009) 175-191.


\bibitem{RS1} M. Reed and B. Simon, {Methods of mathematical physics. Functional analysis}, Academic Press (1980).

\bibitem{Il} A.M. Il'in, Matching of asymptotic expansions of solutions of boundary value problems, AMS, Providence,
Rh (1992).

\bibitem{MM} N.N. Bogolyubov and Yu.A.Mitropol'sk\u{\i}, Asymptotics Methods in Theory of Nonlinear Oscillations, Gordon and Breach,
New York (1962).

\end{thebibliography}
\end{document}